\documentclass[11pt, reqno]{amsart}

\usepackage{fullpage}
\usepackage{amsfonts}
\usepackage{amsmath}
\usepackage{amssymb}
\usepackage{amscd}
\usepackage[mathscr]{eucal}
\usepackage{url}

\allowdisplaybreaks[1]

\def\emptyline{\vspace{12pt}}

\newcommand{\ph}[2]{{\left({#1}\right)}_{#2}}

\renewcommand*{\bar}{\overline}
\newcommand{\gfp}[1]{\Gamma_p{\left({#1}\right)}}
\newcommand{\biggfp}[1]{\Gamma_p{\bigl({#1}\bigr)}}

\newcommand{\bin}[2]{\left({\genfrac{}{}{0pt}{}{#1}{#2}}\right)}

\newcommand{\biggbin}[2]{\biggl({\genfrac{}{}{0pt}{}{#1}{#2}}\biggr)}

\newcommand{\ffs}{\genfrac{}{}{0pt}{}{}{.}}
\newcommand{\fffs}{\genfrac{}{}{0pt}{}{}{\ffs}}
\newcommand{\fc}{\genfrac{}{}{0pt}{}{}{,}}

\DeclareMathOperator{\rep}{rep}
\DeclareMathOperator{\Bin}{Bin}

\theoremstyle{plain}
\newtheorem{theorem}{Theorem}[section]
\newtheorem{lemma}[theorem]{Lemma}
\newtheorem{prop}[theorem]{Proposition}
\newtheorem{cor}[theorem]{Corollary}
\theoremstyle{definition}
\newtheorem{defi}[theorem]{Definition}

\newtheorem{conj}[theorem]{Conjecture}

\numberwithin{equation}{section}

\makeatletter
\def\imod#1{\allowbreak\mkern5mu({\operator@font mod}\,\,#1)}
\makeatother

\begin{document}

\title[Extending Gaussian hypergeometric series to the $p$-adic setting]{Extending Gaussian hypergeometric\\ series to the $p$-adic setting}
\author{Dermot M\lowercase{c}Carthy}  

\address{Department of Mathematics, Texas A\&M University, College Station, TX 77843-3368, USA}

\email{mccarthy@math.tamu.edu}


\begin{abstract}
We define a function which extends Gaussian hypergeometric series to the $p$-adic setting. This new function allows results involving Gaussian hypergeometric series to be extended to a wider class of primes. We demonstrate this by providing various congruences between the function and truncated classical hypergeometric series. These congruences provide a framework for proving the supercongruence conjectures of Rodriguez-Villegas.
\end{abstract}

\maketitle

\section{Introduction}
\makeatletter{\renewcommand*{\@makefnmark}{}
\footnotetext{This work was supported by the UCD Ad Astra Research Scholarship program.}
In \cite{G}, Greene introduced hypergeometric series over finite fields or \emph{Gaussian hypergeometric series}. Let $\mathbb{F}_{p}$ denote the finite field with $p$, a prime, elements. We extend the domain of all multiplicative characters $\chi$ of $\mathbb{F}^{*}_{p}$ to $\mathbb{F}_{p}$, by defining $\chi(0):=0$ (including the trivial character $\varepsilon$) and denote $\bar{B}$ as the inverse of $B$. For characters $A$ and $B$ of $\mathbb{F}_{p}^*$, define
\begin{equation*}\label{FF_Binomial}
\binom{A}{B} := 
%\frac{B(-1)}{p} J(A, \bar{B})= 
\frac{B(-1)}{p} \sum_{x \in \mathbb{F}_{p}} A(x) \bar{B}(1-x).
\end{equation*}
For characters $A_0,A_1,\dotsc, A_n$ and $B_1, \dotsc, B_n$ of $\mathbb{F}_{p}^*$ and 
$x \in \mathbb{F}_{p}$, define the Gaussian hypergeometric series by
\begin{equation*}
{_{n+1}F_n} {\left( \begin{array}{cccc} A_0, & A_1, & \dotsc, & A_n \\
\phantom{A_0} & B_1, & \dotsc, & B_n \end{array}
\Big| \; x \right)}_{p}
:= \frac{p}{p-1} \sum_{\chi} \binom{A_0 \chi}{\chi} \binom{A_1 \chi}{B_1 \chi}
\dotsm \binom{A_n \chi}{B_n \chi} \chi(x)
\end{equation*}
where the sum is over all multiplicative characters $\chi$ of $\mathbb{F}_{p}^*$.

These series are analogous to classical hypergeometric series and have been used in character sum evaluations \cite{GS}, finite field versions of the Lagrange inversion formula \cite{G2}, the representation theory of SL($2, \mathbb{R}$) \cite{G3}, formula for traces of Hecke operators \cite{F2, FOP}, formulas for Ramanujan's $\tau$-function \cite{F2, P}, and evaluating the number of points over $\mathbb{F}_{p}$ of certain algebraic varieties \cite{AO, F2, V}.

They have also played an important role in the proof of many supercongruences \cite{A, AO, K, M1, M2, M, MO}. The main approach was originally established by Ahlgren and Ono in proving the Ap\'ery number supercongruence \cite{AO}.
For $n\geq0$, let $A(n)$ be the the numbers defined by 
\begin{equation*}
A(n):= \sum_{j=0}^{n} \bin{n+j}{j}^2 \bin{n}{j}^2.
\end{equation*}
These numbers were used by Ap\'ery in his proof of the irrationality of $\zeta(3)$ \cite{Ap, vdP} and are commonly known as the Ap\'ery numbers. If we define integers $\gamma(n)$ by
\begin{equation*}
\sum_{n=1}^{\infty} \gamma(n)q^n := q \prod_{n=1}^{\infty}(1-q^{2n})^4 (1-q^{4n})^4,
\end{equation*}
then Beukers conjectured \cite{B1} that for $p\geq3$ a prime,
\begin{equation}\label{for_ANS}
A\left(\tfrac{p-1}{2}\right)\equiv \gamma(p) \pmod{p^2}.
\end{equation}
For a complex number $a$ and a non-negative integer $n$ let $\ph{a}{n}$ denote the rising factorial defined by
\begin{equation}\label{RisFact}
\ph{a}{0}:=1 \quad \textup{and} \quad \ph{a}{n} := a(a+1)(a+2)\dotsm(a+n-1) \textup{ for } n>0.
\end{equation}
Then, for complex numbers $a_i$, $b_j$ and $z$, with none of the $b_j$ being negative integers or zero, we define the truncated generalized hypergeometric series
\begin{equation*}
{{_rF_s} \left[ \begin{array}{ccccc} a_1, & a_2, & a_3, & \dotsc, & a_r \vspace{.05in}\\
\phantom{a_1} & b_1, & b_2, & \dotsc, & b_s \end{array}
\Big| \; z \right]}_{m}
:=\sum^{m}_{n=0}
\frac{\ph{a_1}{n} \ph{a_2}{n} \ph{a_3}{n} \dotsm \ph{a_r}{n}}
{\ph{b_1}{n} \ph{b_2}{n} \dotsm \ph{b_s}{n}}
\; \frac{z^n}{{n!}}.
\end{equation*}
If we first recognize that
\begin{equation*}
A\left(\tfrac{p-1}{2}\right) \equiv 
{_{4}F_3} \Biggl[ \begin{array}{cccc} \frac{1}{2}, & \frac{1}{2}, & \frac{1}{2}, & \frac{1}{2}\vspace{.05in}\\
\phantom{\frac{1}{2}} & 1, & 1, & 1 \end{array}
\bigg| \; 1 \Biggr]_{p-1}
\pmod{p^2}
\end{equation*}
then we can summarize Ahlgren and Ono's proof of (\ref{for_ANS}) in the following two results.
\begin{theorem}[Ahlgren and Ono \cite{AO}]\label{thm_AO1} 
If $p$ is an odd prime and $\phi$ is the character of order 2 of $\mathbb{F}_p^*$, then
\begin{equation*}
{_{4}F_3} \Biggl[ \begin{array}{cccc} \frac{1}{2}, & \frac{1}{2}, & \frac{1}{2}, & \frac{1}{2}\vspace{.05in}\\
\phantom{\frac{1}{2}} & 1, & 1, & 1 \end{array}
\bigg| \; 1 \Biggr]_{p-1}
\equiv
-p^3 \:  {_{4}F_3}  \Biggl( \begin{array}{cccc} \phi, & \phi, & \phi, & \phi \\
\phantom{\phi} & \varepsilon, & \varepsilon, & \varepsilon \end{array}
\bigg| \; 1 \Biggr)_{p}-p \pmod{p^2} .
\end{equation*}
\end{theorem}
\begin{theorem}[Ahlgren and Ono \cite{AO}] 
If $p$ is an odd prime and $\phi$ is the character of order 2 of $\mathbb{F}_p^*$, then
\begin{equation*}
-p^3 \:  {_{4}F_3}  \Biggl( \begin{array}{cccc} \phi, & \phi, & \phi, & \phi \\
\phantom{\phi} & \varepsilon, & \varepsilon, & \varepsilon \end{array}
\bigg| \; 1 \Biggr)_{p}-p
=\gamma(p) .
\end{equation*}
\end{theorem}
\noindent It is interesting to note the analogue between the parameters in the hypergeometric series in Theorem~\ref{thm_AO1}. ``One over two'' has been replaced with a character of order 2 and ``one over one'' has been replaced with a character of order one. 
This approach of using the Gaussian hypergeometric series as an intermediate step has since become the template for proving these types of supercongruences and the analogue between the parameters of the generalized and Gaussian hypergeometric series has also been in evidence in these cases.

In \cite{R} Rodriguez-Villegas examined the relationship between the number of points over $\mathbb{F}_p$ on certain Calabi-Yau manifolds and truncated hypergeometric series which correspond to a particular period of the manifold. 
In doing so, he identified numerically 22 possible supercongruences. 18 of these relate truncated generalized hypergeometric series to the $p$-th Fourier coefficient of certain modular forms via modulo $p^2$ and $p^3$ congruences. 
Two of the 18 have been proven outright \cite{A, K} with three more established for primes in a particular congruence class and up to sign otherwise \cite{M}, all using Gaussian hypergeometric series as an intermediate step.
 (We note that the case proved in \cite{A} had previously been established in \cite{VH} by other means.) We now consider one of the outstanding conjectures of Rodriguez-Villegas. Let 
\begin{equation}\label{for_ModForm}
f(z):= f_1(z)+5f_2(z)+20f_3(z)+25f_4(z)+25f_5(z)=\sum_{n=1}^{\infty} c(n) q^n
\end{equation}
where $f_i(z):=\eta^{5-i}(z) \hspace{2pt} \eta^4(5z) \hspace{2pt} \eta^{i-1}(25z)$, $\eta(z):=q^{\frac{1}{24}} \prod_{n=1}^{\infty}(1-q^n)$ is the Dedekind eta function and $q:=e^{2 \pi i z}$. Then $f$ is a cusp form of weight four on the congruence subgroup $\Gamma_0(25)$ and we have the following conjecture. 
\begin{conj}[Rodriguez-Villegas \cite{R}]\label{thm_DMCMain}
If $p \neq 5$ is prime and $c(p)$ is as defined in (\ref{for_ModForm}), then 
\begin{equation}\label{RV_Conj}
{_4F_3} \Biggl[ \begin{array}{cccc} \frac{1}{5}, & \frac{2}{5}, & \frac{3}{5}, & \frac{4}{5} \\
\phantom{\frac{1}{5},} & 1, & 1, & 1\end{array}
\bigg| \; 1 \Biggr]_{p-1}
\equiv c(p) \pmod {p^3}.
\end{equation}
\end{conj}
\noindent Using the approach of Ahlgren and Ono we would expect to be able to relate the truncated hypergeometric series on the left hand side of (\ref{RV_Conj}) to the Gaussian hypergeometric series
\begin{equation*}
_4F_3 \Biggl( \begin{array}{cccc} \chi_5, & \chi_5^2, & \chi_5^3, & \chi_5^4 \\
\phantom{\chi_5} & \varepsilon, & \varepsilon, & \varepsilon \end{array}
\bigg| \; 1 \Biggr)_p,
\end{equation*}
where $\chi_5$ is a character of order $5$ of $\mathbb{F}_p^*$. However this series is only defined for $p \equiv 1 \pmod 5$ which would restrict any eventual results. A similar restriction was also encountered by the author of \cite{M}. 
This issue did not affect the authors of \cite{A, K} as all top line parameters in their cases were characters of order 2 thus restricting the series to odd primes which was all that was required. Many of the other applications of these series also encounter these restrictions. 
For example the results in \cite{F2} are restricted to primes congruent to $1$ modulo $12$. We would like to develop some generalization of Gaussian hypergeometric series which does not have such restrictions implicit in its definition. 
We achieve this by reformulating the series as an expression in terms of the $p$-adic gamma function which we can then extend in the $p$-adic setting.

\section{Statement of Results}

The Gaussian hypergeometric series which occur in the applications of the series referenced in Section 1 have all been of a certain form. All top line characters, $A_i$, have been non-trivial, all bottom line characters, $B_j$, have been trivial and the argument $x$ has equaled 1. The Gaussian hypergeometric series analogous to the classical series in the supercongruence conjectures of Rodriguez-Villegas would also be of this form. 

Therefore, our starting point for extending Gaussian hypergeometric series is to examine series of this type. In a similar manner to \cite{M2, M}, using the relationship between Jacobi and Gauss sums, and the Gross-Koblitz formula \cite{GK}, we can express the series in terms of the $p$-adic gamma function. 
For a positive integer $n$, let $m_1, m_2, \dots, m_{n+1}, d_1, d_2, \dots, d_{n+1}$ also be positive integers, such that $0<\frac{m_i}{d_i}< 1$ for $1 \leq i \leq n+1$. 
Let $p$ be a prime such that $p \equiv 1 \pmod {d_i}$ for each $i$ and let $\rho_i$ be the character of order $d_i$ of $\mathbb{F}_p^*$ given by $\bar{\omega}^{\frac{p-1}{d_i}}$, where $\omega$ is the Teichm\"{u}ller character. 
Without loss of generality we can assume $0<\frac{m_1}{d_1} \leq \frac{m_2}{d_2} \leq \dots \leq \frac{m_{n+1}}{d_{n+1}}<1$. We define $r_i:=\frac{p-1}{d_i}$ for brevity. 
Then we also define $m_0:=-1$, $m_{n+2}:=p-2$ and $d_0=d_{n+2}:=p-1$ so that $m_0r_0=-1$ and $m_{n+2}r_{n+2}=p-2$. If $\gfp{\cdot}$ is the $p$-adic gamma function, then we have
\begin{align}\label{for_GHStoP}
&(-1)^n \: p^n \:  {_{n+1}F_n}  {\left( \begin{array}{cccc} \rho_1^{m_1}, & \rho_2^{m_2}, & \dotsc, & \rho_{n+1}^{m_{n+1}} \\
\phantom{\rho_1^{m_1}} & \varepsilon, & \dotsc, & \varepsilon \end{array}
\Big| \; 1 \right)}_{p}\\
\notag &= \frac{-1}{p-1} \sum_{k=0}^{n+1} (-p)^k \sum_{j=m_kr_k+1}^{m_{k+1}r_{k+1}} \omega^{j(n+1)}(-1) 
{\biggfp{\tfrac{j}{p-1}}}^{n+1}
\prod_{\substack{i=1\\i>k}}^{n+1} \frac{\biggfp{\frac{m_i}{d_i}-\frac{j}{p-1}}}{\biggfp{\frac{m_i}{d_i}}}
\prod_{\substack{i=1\\i\leq k}}^{n+1} \frac{\biggfp{\frac{d_i+m_i}{d_i}-\frac{j}{p-1}}}{\biggfp{\frac{m_i}{d_i}}}.
\end{align}

Let $\left\lfloor x \right\rfloor$ denote the greatest integer less than or equal to $x$ and let $\langle x \rangle$ be the fractional part of $x$, i.e. $x- \left\lfloor x \right\rfloor$. We then define the following generalization.
\begin{defi}\label{def_GFn}
Let $p$ be an odd prime and let $n \in \mathbb{Z}^{+}$. For $1 \leq i \leq n+1$, let $\frac{m_i}{d_i} \in \mathbb{Q} \cap \mathbb{Z}_p$ such that $0<\frac{m_i}{d_i}<1$. Then define
\begin{equation*}
{_{n+1}G} \left( \tfrac{m_1}{d_1}, \tfrac{m_2}{d_2}, \dotsc, \tfrac{m_{n+1}}{d_{n+1}} \right)_p
:= \frac{-1}{p-1}  \sum_{j=0}^{p-2} 
%(-1)^{j(n+1)} 
{\left((-1)^j\biggfp{\tfrac{j}{p-1}}\right)}^{n+1} 
\prod_{i=1}^{n+1} \frac{\biggfp{\langle \frac{m_i}{d_i}-\frac{j}{p-1}\rangle}}{\biggfp{\frac{m_i}{d_i}}}
(-p)^{-\lfloor{\frac{m_i}{d_i}-\frac{j}{p-1}\rfloor}}.
\end{equation*}
\end{defi}
\noindent
We first note that ${_{n+1}G} \in \mathbb{Z}_p$. We also note that the function is defined at all primes not dividing the $d_i$. Combining (\ref{for_GHStoP}) and Definition \ref{def_GFn}, one can easily see that we recover the Gaussian hypergeometric series via the following result.
\begin{prop}\label{prop_GtoGHS}
Let $n \in \mathbb{Z}^{+}$ and, for $1 \leq i \leq n+1$, let $\frac{m_i}{d_i} \in \mathbb{Q}$ such that $0<\frac{m_i}{d_i}<1$. Let $p \equiv 1 \pmod {d_i}$, for each $i$, be prime and let $\rho_i$ be the character of order $d_i$ of $\mathbb{F}_p^*$ given by $\bar{\omega}^{\frac{p-1}{d_i}}$. Then
\begin{align*}
{_{n+1}G} \left( \tfrac{m_1}{d_1}, \tfrac{m_2}{d_2}, \dotsc, \tfrac{m_{n+1}}{d_{n+1}} \right)_p
=(-1)^n \: p^n \:  {_{n+1}F_n}  {\Bigg( \begin{array}{cccc} \rho_1^{m_1}, & \rho_2^{m_2}, & \dotsc, & \rho_{n+1}^{m_{n+1}} \\
\phantom{\rho_1^{m_1}} & \varepsilon, & \dotsc, & \varepsilon \end{array}
\bigg| \; 1 \Biggr)}_{p} \; \fffs
\end{align*}
\end{prop}

\noindent 
Using this proposition, the $_{n+1}G$ function can be used in place of Greene's function in most applications and should allow results to be extended to a wider class of primes in many cases. We demonstrate this here by considering its relationship with the classical series. 

One of the main results offering congruences between generalized and Gaussian hypergeometric series has been Theorem 1 in \cite{M}. This theorem relates the same Gaussian hypergeometric series as appears on the right hand side of Proposition \ref{prop_GtoGHS} to a truncated generalized hypergeometric series via a modulo $p^2$ congruence. Therefore, applying Proposition \ref{prop_GtoGHS} to Theorem 1 in \cite{M} we get the following result.
\begin{theorem}\label{thm_DMCMortTrunc}
Let $n \in \mathbb{Z}^{+}$ and, for $1 \leq i \leq n+1$, let $\frac{m_i}{d_i} \in \mathbb{Q}$ such that $0<\frac{m_i}{d_i}<1$. Let $p \equiv 1 \pmod {d_i}$, for each $i$, be prime and let $\rho_i$ be the character of order $d_i$ of $\mathbb{F}_p^*$ given by $\bar{\omega}^{\frac{p-1}{d_i}}$. If $S:=\sum_{i=1}^{n+1} \frac{m_i}{d_i} \geq n-1$ and $\delta:=\prod_{i=1}^{n+1} \biggfp{1-\frac{m_i}{d_i}}$ when $S=n-1$ and zero otherwise, then
\begin{equation*}
{_{n+1}G} \left( \tfrac{m_1}{d_1}, \tfrac{m_2}{d_2}, \dotsc, \tfrac{m_{n+1}}{d_{n+1}} \right)_p
\equiv  
{_{n+1}F_n} \Biggl[ \begin{array}{cccc} \frac{m_1}{d_1}, & \frac{m_2}{d_2}, & \dotsc, & \frac{m_{n+1}}{d_{n+1}} \vspace{.05in}\\
\phantom{\frac{m_1}{d_1}} & 1, & \dotsc, & 1 \end{array}
\bigg| \; 1 \Biggr]_{p-1}
+ \delta \cdot p \pmod {p^2}.
\end{equation*}
\end{theorem} 

\noindent The main results of this paper establish $4$ families of congruences between the $_{n+1}G$ function and truncated generalized hypergeometric series which extend Theorem \ref{thm_DMCMortTrunc} to primes in other congruence classes, as follows.
\begin{theorem}\label{thm_2G}
Let $2 \leq d \in \mathbb{Z}$ and let $p$ be an odd prime such that $p\equiv \pm1 \imod d$. Then
\begin{align*}
{_{2}G} \Bigl( \tfrac{1}{d}, 1-\tfrac{1}{d} \Bigr)_p
&\equiv
{_{2}F_1} \Biggl[ \begin{array}{cc} \frac{1}{d}, & 1-\frac{1}{d}\vspace{.05in}\\
\phantom{\frac{1}{d}} & 1 \end{array}
\bigg| \; 1 \Biggr]_{p-1}
\pmod {p^2} .
\end{align*}
\end{theorem}

\begin{theorem}\label{thm_3G}
Let $2 \leq d \in \mathbb{Z}$ and let $p$ be an odd prime such that $p\equiv \pm1 \imod d$. Then
\begin{align*}
{_{3}G} \Bigl( \tfrac{1}{2}, \tfrac{1}{d} , 1-\tfrac{1}{d} \Bigr)_p
&\equiv
{_{3}F_2} \Biggl[ \begin{array}{ccc} \frac{1}{2}, & \frac{1}{d}, & 1-\frac{1}{d}\vspace{.05in}\\
\phantom{\frac{1}{d}} & 1, &1 \end{array}
\bigg| \; 1 \Biggr]_{p-1}
\pmod {p^2} .
\end{align*}
\end{theorem}

\begin{theorem}\label{thm_4G1}
Let $d_1, d_2 \geq 2$ be integers and let $p$ be an odd prime such that $p\equiv \pm1 \imod {d_1}$ and $p\equiv \pm1 \imod {d_2}$. If $s(p):=\biggfp{\frac{1}{d_1}}\biggfp{\frac{d_1-1}{d_1}}\biggfp{\frac{1}{d_2}}\biggfp{\frac{d_2-1}{d_2}}=(-1)^{\left\lfloor \frac{p-1}{d_1} \right\rfloor+\left\lfloor \frac{p-1}{d_2} \right\rfloor}$, then
\begin{equation*}
{_{4}G} \left(\tfrac{1}{d_1} , 1-\tfrac{1}{d_1}, \tfrac{1}{d_2} , 1-\tfrac{1}{d_2}\right)_p
\equiv
{_{4}F_3} \Biggl[ \begin{array}{cccc} \frac{1}{d_1}, & 1-\frac{1}{d_1}, & \frac{1}{d_2}, & 1-\frac{1}{d_2}\vspace{.05in}\\
\phantom{\frac{1}{d_1}} & 1, & 1, & 1 \end{array}
\bigg| \; 1 \Biggr]_{p-1}
+s(p)\hspace{1pt}p
\pmod {p^3}.
\end{equation*}
\end{theorem}

\begin{theorem}\label{thm_4G2}
Let $r, d \in \mathbb{Z}$ such that $2 \leq r \leq d-2$ and $\gcd(r,d)=1$. Let $p$ be an odd prime such that $p\equiv \pm1 \imod d$ or $p\equiv \pm r \imod d$ with $r^2 \equiv \pm1 \imod d$. If $s(p) := \gfp{\tfrac{1}{d}} \gfp{\tfrac{r}{d}}\gfp{\tfrac{d-r}{d}}\gfp{\tfrac{d-1}{d}}$, then
\begin{align*}
{_{4}G} \Bigl(\tfrac{1}{d} , \tfrac{r}{d}, 1-\tfrac{r}{d} , 1-\tfrac{1}{d}\Bigr)_p
&\equiv
{_{4}F_3} \Biggl[ \begin{array}{cccc} \frac{1}{d}, & \frac{r}{d}, & 1-\frac{r}{d}, & 1-\frac{1}{d}\vspace{.05in}\\
\phantom{\frac{1}{d_1}} & 1, & 1, & 1 \end{array}
\bigg| \; 1 \Biggr]_{p-1}
+s(p)\hspace{1pt} p
\pmod {p^3}.
\end{align*}
\end{theorem}

\noindent Theorems \ref{thm_2G} and \ref{thm_3G} extend Theorem \ref{thm_DMCMortTrunc}, when $n=1,2$, to primes in an additional congruence class. 
However, the price of this extension is a loss in some generality of the arguments of the series. 
In the case $n=3$, Theorems \ref{thm_4G1} and \ref{thm_4G2} not only extend Theorem \ref{thm_DMCMortTrunc} to primes in additional congruence classes but also to a modulo $p^3$ relation, which is a significant development, as we will see in the next paragraph. 
Again this extension is accompanied by a loss in generality of the arguments of the series. 
One of the consequences of the methods contained in this paper is that the sum of the arguments of $_{n+1}G$ in any extension will equal $\frac{n+1}{2}$. 
This means that the condition on $S$ in Theorem \ref{thm_DMCMortTrunc} will only be satisfied if $n \leq 3$. Therefore, extension of Theorem \ref{thm_DMCMortTrunc} to primes beyond $p \equiv 1 \pmod {d_i}$ when $n>3$ is not possible using current methods. 
Numerical testing would also suggest that such a general formula does not exist. 

The truncated generalized hypergeometric series appearing in Theorems \ref{thm_2G} to \ref{thm_4G2} include all 22 truncated hypergeometric series occurring in the supercongruence conjectures of Rodriguez-Villegas \cite{R}. 
Furthermore, in each of these cases, the congruences in Theorems \ref{thm_2G} to \ref{thm_4G2} hold for all primes required in the conjectures, due to the particular parameters involved. 
Thus, these congruences provide a framework for proving all 22 cases. In fact, in \cite{DMC2} we use Theorem \ref{thm_4G2} with $d=5$ to prove Conjecture \ref{thm_DMCMain}.

Using Proposition \ref{prop_GtoGHS}, it is easy to see the following corollaries of Theorems \ref{thm_2G} to \ref{thm_4G2}.
\begin{cor}\label{cor_2G}
Let $2 \leq d \in \mathbb{Z}$ and let $p$ be a prime such that $p\equiv 1 \pmod d$. If $\rho$ is the character of order $d$ of $\mathbb{F}_p^*$ given by $\bar{\omega}^{\frac{p-1}{d}}$, then
\begin{align*}
-p \:  {_{2}F_1}  {\Biggl( \begin{array}{cc} \rho, & \bar{\rho} \\
\phantom{\rho} & \varepsilon \end{array}
\bigg| \; 1 \Biggr)}_{p}
\equiv
{_{2}F_1} \Biggl[ \begin{array}{cc} \frac{1}{d}, & 1-\frac{1}{d}\vspace{.05in}\\
\phantom{\frac{1}{d}} & 1 \end{array}
\bigg| \; 1 \Biggr]_{p-1}
\pmod {p^2}.
\end{align*}
\end{cor}

\begin{cor}\label{cor_3G}
Let $2 \leq d \in \mathbb{Z}$ and let $p$ be a prime such that $p\equiv 1 \pmod d$. If $\psi$ is the character of order 2 and $\rho$ is the character of order $d$ of $\mathbb{F}_p^*$ given by $\bar{\omega}^{\frac{p-1}{d}}$, then
\begin{align*}
p^2 \:  {_{3}F_2}  {\Biggl( \begin{array}{ccc} \psi, & \rho, & \bar{\rho} \\
\phantom{\psi} & \varepsilon, & \varepsilon \end{array}
\bigg| \; 1 \Biggr)}_{p}
\equiv
{_{3}F_2} \Biggl[ \begin{array}{ccc} \frac{1}{2}, & \frac{1}{d}, & 1-\frac{1}{d}\vspace{.05in}\\
\phantom{\frac{1}{d}} & 1, &1 \end{array}
\bigg| \; 1 \Biggr]_{p-1}
\pmod {p^2}.
\end{align*}
\end{cor}

\begin{cor}\label{cor_4G1}
Let $2 \leq d_1, d_2 \in \mathbb{Z}$ and let $p$ be a prime such that $p\equiv~1~{\pmod {d_1}}$ and $p\equiv 1 \pmod {d_2}$. If $s(p):=\biggfp{\frac{1}{d_1}}\biggfp{\frac{d_1-1}{d_1}}\biggfp{\frac{1}{d_2}}\biggfp{\frac{d_2-1}{d_2}}=(-1)^{\left\lfloor \frac{p-1}{d_1} \right\rfloor+\left\lfloor \frac{p-1}{d_2} \right\rfloor}$ and $\rho_i$ is the character of order $d_i$ of $\mathbb{F}_p^*$ given by $\bar{\omega}^{\frac{p-1}{d_i}}$, then
\begin{equation*}
-p^3 \:  {_{4}F_3}  {\Biggl( \begin{array}{cccc} \rho_1, & \bar{\rho_1}, & \rho_2, & \bar{\rho_2} \\
\phantom{\rho_1} & \varepsilon, & \varepsilon, & \varepsilon \end{array}
\bigg| \; 1 \Biggr)}_{p}
\equiv
{_{4}F_3} \Biggl[ \begin{array}{cccc} \frac{1}{d_1}, & 1-\frac{1}{d_1}, & \frac{1}{d_2}, & 1-\frac{1}{d_2}\vspace{.05in}\\
\phantom{\frac{1}{d_1}} & 1, & 1, & 1 \end{array}
\bigg| \; 1 \Biggr]_{p-1}
+s(p)\hspace{1pt}p
\pmod {p^3}.
\end{equation*}
\end{cor}

\begin{cor}\label{cor_4G2}
Let $r, d \in \mathbb{Z}$ such that $2 \leq r \leq d-2$ and $\gcd(r,d)=1$. Let $p$ be a prime such that $p\equiv 1 \pmod d$. If $s(p) := \gfp{\tfrac{1}{d}} \gfp{\tfrac{r}{d}}\gfp{\tfrac{d-r}{d}}\gfp{\tfrac{d-1}{d}}$ and $\rho$ is the character of order $d$ of $\mathbb{F}_p^*$ given by $\bar{\omega}^{\frac{p-1}{d_i}}$, then
\begin{equation*}
-p^3 \:  {_{4}F_3}  {\Biggl( \begin{array}{cccc} \rho, & \bar{\rho}, & \rho^r, & \bar{\rho}^r \\
\phantom{\rho_1} & \varepsilon, & \varepsilon, & \varepsilon \end{array}
\bigg| \; 1 \Biggr)}_{p}
\equiv
{_{4}F_3} \Biggl[ \begin{array}{cccc} \frac{1}{d}, & \frac{r}{d}, & 1-\frac{r}{d}, & 1-\frac{1}{d}\vspace{.05in}\\
\phantom{\frac{1}{d_1}} & 1, & 1, & 1 \end{array}
\bigg| \; 1 \Biggr]_{p-1}
+s(p)\hspace{1pt} p
\pmod {p^3}.
\end{equation*}
\end{cor}
%\section*{Remark}
\noindent We note Corollaries \ref{cor_2G} and \ref{cor_3G} are special cases of Theorem 1 in \cite{M}. However Corollaries \ref{cor_4G1} and \ref{cor_4G2} are new and are the first general modulo $p^3$ results in this area.

The remainder of the paper is organized as follows. In Section 3 we recall some properties of the $p$-adic gamma function and its logarithmic derivatives and we also develop some preliminary results for later use. Section 4 deals with the proofs of Theorems \ref{thm_2G} to \ref{thm_4G2}.

%%%%%%%%%%%%%%%%%%%%%%%%%%%%%%%%%%%%%%%%%%%%%%%%%%%
%%%%%%%%%%%%%%%%%%%%%%%%%%%%%%%%%%%%%%%%%%%%%%%%%%%
%%%%%%%%%%%%%%%%%%%%%%%%%%%%%%%%%%%%%%%%%%%%%%%%%%%
%%%%%%%%%%%%%%%%%%%%%%%%%%%%%%%%%%%%%%%%%%%%%%%%%%%

\section{Preliminaries}
\subsection{$p$-adic preliminaries}
We first recall the $p$-adic gamma function. For further details, see \cite{Ko}. 
Let $p$ be an odd prime.  For $n \in \mathbb{Z}^{+}$ we define the $p$-adic gamma function as
\begin{align*}
\gfp{n} &:= {(-1)}^n \prod_{\substack{0<j<n\\p \nmid j}} j \\
\intertext{and extend to all $x \in\mathbb{Z}_p$ by setting $\gfp{0}:=1$ and} 
\gfp{x} &:= \lim_{n \rightarrow x} \gfp{n}
\end{align*}
for $x\neq 0$, where $n$ runs through any sequence of positive integers $p$-adically approaching $x$. 
This limit exists, is independent of how $n$ approaches $x$, and determines a continuous function
on $\mathbb{Z}_p$ with values in $\mathbb{Z}^{*}_p$.
We now state some basic properties of the $p$-adic gamma function.
\begin{prop}[\cite{Ko} Chapter II.6]\label{prop_pGamma}
Let $x, y \in \mathbb{Z}_{p}$ and $n \in \mathbb{Z}^{+}$. Then\\[6pt]
\textup{(1)}
$\gfp{x+1}=
\begin{cases}
-x \hspace{1pt} {\gfp{x}} & \quad \text{if } x \in \mathbb{Z}_p^* \, ,\\
- {\gfp{x}} & \quad \text{otherwise}.
\end{cases}$
\\[6pt]

\noindent \textup{(2)}
$\Gamma_p(x)\Gamma_p(1-x) = {(-1)}^{x_0}$,
where $x_0 \in \{1,2, \dotsc, {p}\}$ satisfies $x_0 \equiv x \pmod {p}$.
\\

\noindent \textup{(3)}
If $x \equiv y \pmod{p^n}$, then $\gfp{x} \equiv \gfp{y} \pmod{p^n}$.
\end{prop}

\noindent
We also consider the logarithmic derivatives of $\Gamma_p$. For $x \in \mathbb{Z}_{p}$, define
$$G_1(x):= \dfrac{\Gamma_p^{\prime}(x)}{\gfp{x}} \qquad \textup{and} \qquad
 G_2(x):= \dfrac{\Gamma_p^{\prime \prime}(x)}{\gfp{x}}.$$ These also satisfy some basic properties  which we state below. Some of these results can be found in \cite{AO}, \cite{CDE} and \cite{K}. If not, we include a short proof.

\begin{prop}\label{prop_pGammaG}
Let $x \in \mathbb{Z}_{p}$. Then\\[6pt]
\textup{(1)}
$G_1(x+1) - G_1(x) =
\begin{cases}
 1/ x & \quad \text{if } x \in \mathbb{Z}_p^* \, ,\\
0 & \quad \text{otherwise}.
\end{cases}$
\\[0pt]

\noindent \textup{(2)}
$G_1(x+1)^2 - G_2(x+1) - G_1(x)^2 + G_2(x)=
\begin{cases}
 1/ x^2 & \quad \text{if } x \in \mathbb{Z}_p^*\, ,\\
0 & \quad \text{otherwise}.
\end{cases}$
%\\[0pt]

\noindent \textup{(3)}
$G_1(x) = G_1(1-x)$.
%\\[0pt]

\noindent \textup{(4)}
$G_1(x)^2 - G_2(x) = - G_1(1-x)^2 + G_2(1-x)$.\\
\end{prop}

\begin{proof}
(1) and (3) are obtained from differentiating the results in Proposition~\ref{prop_pGamma} (1) and (2) respectively, while (2) and (4) follow from differentiating (1) and (3).
\end{proof}

\noindent We also have some congruence properties of the $p$-adic gamma function and its logarithmic derivatives as follows.
\begin{prop}\label{prop_pGammaCong}
Let $p \geq 7$ be a prime, $x \in \mathbb{Z}_{p}$ and $z \in p  \hspace{1pt} \mathbb{Z}_{p}$. Then\\[6pt]
\textup{(1)}
$G_1(x)$, $G_2(x) \in \mathbb{Z}_p$.
%\\[-3pt]

\noindent \textup{(2)}
$\Gamma_p(x+z) \equiv \Gamma_p(x) \left(1+zG_1(x) +\frac{z^2}{2} G_2(x) \right) \pmod{p^3}$.
%\\[-3pt]

\noindent \textup{(3)}
$\Gamma_p^{\prime}(x+z) \equiv \Gamma_p^{\prime}(x) +z \Gamma_p^{\prime \prime} (x) \pmod{p^2}$.
\end{prop}

\begin{proof}
See \cite{K} Proposition 2.3.
\end{proof}

\begin{cor}\label{cor_pGammaCong1}
Let $p \geq 7$ be a prime, $x \in \mathbb{Z}_{p}$ and $z \in p  \hspace{1pt} \mathbb{Z}_{p}$. Then\\[6pt]
\noindent \textup{(1)}
$\Gamma_p^{\prime}(x+z) \equiv \Gamma_p^{\prime}(x)  \pmod{p}$.
%\\[-3pt]

\noindent \textup{(2)}
$\Gamma_p^{\prime\prime}(x+z) \equiv \Gamma_p^{\prime\prime}(x)  \pmod{p}$.
%\\[-3pt]

\noindent \textup{(3)}
$G_1(x+z) \equiv G_1(x)  \pmod{p}$.
%\\[-3pt]

\noindent \textup{(4)}
$G_2(x+z) \equiv G_2(x)  \pmod{p}$.
\end{cor}

\begin{proof}
By definition, $\Gamma_p(x) \in \mathbb{Z}_p^{*}$. Thus, by Proposition \ref{prop_pGammaCong} (1) and the definitions of $G_1(x)$ and $G_2(x)$, we see that $\Gamma_p^{\prime}(x)$ and $\Gamma_p^{\prime\prime}(x) \in \mathbb{Z}_p$.
Observe that (1) then follows from Proposition~\ref{prop_pGammaCong} (3).
For (2), one uses similar methods to Proposition~2.3 in \cite{K}.
Finally, (3) and (4) follow from (1)  and (2) and the definitions of $G_1(x)$ and $G_2(x)$.
\end{proof}

\begin{cor}\label{cor_pGammaCong2}
Let $p \geq 7$ be a prime, $x \in \mathbb{Z}_{p}$ and $z \in p  \hspace{1pt} \mathbb{Z}_{p}$. Then
\begin{equation*}
G_1(x)\equiv G_1(x+z)+z\left(G_1(x+z)^2 - G_2(x+z)\right)
\pmod{p^2}.
\end{equation*}
\end{cor}

\begin{proof}
By Proposition \ref{prop_pGammaCong}, we see that
\begin{align*}
G_1(x)= \frac{\Gamma_p^{\prime}(x)}{\gfp{x}}
\equiv \frac{\Gamma_p^{\prime}(x+z) - z \Gamma_p^{\prime \prime} (x)}{\Gamma_p(x+z) - z \Gamma_p^{\prime}(x)}
\pmod{p^2}.
\end{align*}
Multiplying the numerator and denominator by ${\Gamma_p(x+z) + z \Gamma_p^{\prime}(x)}$ we get that
\begin{align*}
G_1(x)
&\equiv \frac{
\Gamma_p^{\prime}(x+z) \Gamma_p(x+z) + z 
\left(\Gamma_p^{\prime}(x) \Gamma_p^{\prime}(x+z)-\Gamma_p(x+z) \Gamma_p^{\prime \prime} (x)\right)
}
{\Gamma_p(x+z)^2}\\
&\equiv \frac{
\Gamma_p^{\prime}(x+z) \Gamma_p(x+z) + z 
\left(\Gamma_p^{\prime}(x+z) \Gamma_p^{\prime}(x+z)-\Gamma_p(x+z) \Gamma_p^{\prime \prime} (x+z)\right)
}
{\Gamma_p(x+z)^2}\\
&\equiv G_1(x+z) +z \left(G_1(x+z)^2-G_2(x+z)\right)
\pmod{p^2}.
\end{align*}
\end{proof}

We now introduce some notation for a $p$-adic integer's basic representative modulo~$p$.
\begin{defi}
For a prime $p$ and $x \in \mathbb{Z}_p$ we define $\rep_p(x) \in \{1, 2, \cdots, p\}$ via the congruence
\begin{equation*}
\rep_p(x) \equiv x \pmod p.
\end{equation*}
\end{defi}
\noindent We will drop the subscript $p$ when it is clear from the context. We have the following basic properties of $\rep(\cdot)$.

\begin{prop}\label{prop_repOneminus}
Let $p$ be a prime and let $x \in \mathbb{Z}_p$. Then
\begin{equation*}
\rep(1-x)=p+1-\rep(x).
\end{equation*}
\end{prop}

\begin{proof}
We first note that $1-x \in \mathbb{Z}_p$ and that $p+1-\rep(x) \in \{1, 2, \cdots, p\}$. We then see that
\begin{equation*}
p+1-\rep(x) \equiv 1-\rep(x) \equiv 1- x \pmod p.
\end{equation*}
\end{proof}

\begin{lemma}\label{cor_repGenformula}
Let $p$ be a prime and let $d \in \mathbb{Z}$ such that $p\equiv a \pmod d$ with ${0<a<d}$. Then
\begin{equation*}
\rep(\tfrac{a}{d})=p-\lfloor \tfrac{p-1}{d}\rfloor
\end{equation*}
and
\begin{equation*}
\rep(\tfrac{d-a}{d})=\lfloor \tfrac{p-1}{d} \rfloor +1 \;.
\end{equation*}
\end{lemma}

\begin{proof}
We first note that $\tfrac{a}{d} \in \mathbb{Z}_p$. It easy to see that $2 \leq p-\lfloor \tfrac{p-1}{d} \rfloor \leq p$ and that
$$p-\lfloor \tfrac{p-1}{d} \rfloor= p - \tfrac{p-a}{d} \equiv \tfrac{a}{d} \pmod{p}.$$
Thus $\rep(\tfrac{a}{d})=p-\lfloor \tfrac{p-1}{d}\rfloor$. The second result follows from Proposition \ref{prop_repOneminus}.
\end{proof}

\noindent We now use these properties to develop further results concerning the $p$-adic gamma function. We recall the definition of the rising factorial $(a)_n$ in (\ref{RisFact}) and allow $a \in \mathbb{Z}_p$.%\newpage
\begin{prop}\label{prop_gammapmd}
Let $p$ be an odd prime and let $x \in \mathbb{Z}_p$. If $0 \leq j \leq p\in \mathbb{Z}$, then
\begin{equation*}
\gfp{x+j}=
\begin{cases}
(-1)^j \; \gfp{x} \, \ph{x}{j} & \qquad \text{if } 0\leq j\leq p- \rep(x),\\[6pt]
(-1)^j \; \gfp{x} \, \ph{x}{j} \, {\left({x+p-\rep(x)}\right)}^{-1} & \qquad \text{if } p- \rep(x)+1\leq j\leq p.\\
\end{cases}
\end{equation*} 
\end{prop}

\begin{proof}
For $j=0$ the result is trivial. Assume $j>0$. For $0\leq k \leq j-1$,
\begin{equation}\label{for_pZp}
x +k \in p \hspace{1pt} \mathbb{Z}_p \Longleftrightarrow \rep(x) +k \in p \hspace{1pt} \mathbb{Z}_p \Longleftrightarrow \rep(x) +k = p \Longleftrightarrow k = p -\rep(x).
\end{equation}
Using Proposition \ref{prop_pGamma} (1) the result follows.
\end{proof}

\noindent For $i$, $n \in \mathbb{Z}^{+}$, we define the generalized harmonic sums, ${H}^{(i)}_{n}$, by
\begin{equation*}
{H}^{(i)}_{n}:= \sum^{n}_{j=1} \frac{1}{j^i}
\end{equation*}
and ${H}^{(i)}_{0}:=0$. We can now use the above to develop some congruences for use in Section 4.

\begin{lemma}\label{lem_ProdGammap}
Let $p$ be an odd prime and let $x \in \mathbb{Z}_p$. If $0 \leq j \leq p\in \mathbb{Z}$, then
\begin{equation*}
\gfp{x+j}\equiv \left(\rep(x)+j-1\right)!{(-1)}^{\rep(x)+j} \cdot \delta
\pmod{p} 
\end{equation*}
where
\begin{equation*}
\delta=
\begin{cases}
1 & \text{if } 0\leq j\leq p- \rep(x),\\[3pt]
\frac{1}{p} & \text{if } p- \rep(x)+1\leq j\leq p.\\[3pt]
\end{cases}
\end{equation*}
\end{lemma}

\begin{proof}
By Proposition \ref{prop_pGamma} (3) we see that
\begin{equation*}
\gfp{x+j} \equiv \gfp{\rep(x)+j} \pmod{p}.
\end{equation*}
Combining Proposition \ref{prop_pGamma} (1) and (\ref{for_pZp}) yields the result.
\end{proof}

\begin{lemma}\label{lem_SumGammaG1}
Let $p\geq 7$ be a prime and let $x \in \mathbb{Z}_p$. If $0 \leq j \leq p-1\in \mathbb{Z}$, then
\begin{equation*}
G_1\left(x+j\right) -G_1(1+j)
\equiv
H_{\rep(x)-1+j}^{(1)}- H_{j}^{(1)}-\delta\\[3pt]
\pmod{p}
\end{equation*}
where 
\begin{equation*}
\delta=
\begin{cases}
0 & \text{if } 0\leq j\leq p- \rep(x),\\[3pt]
\frac{1}{p} & \text{if } p- \rep(x)+1\leq j\leq p-1.\\[3pt]
\end{cases}
\end{equation*}
\end{lemma}

\begin{proof}
By Corollary \ref{cor_pGammaCong1} (3) we see that
\begin{equation*}
G_1\left(x+j\right) -G_1(1+j) \equiv G_1\left(\rep(x)+j\right) -G_1(1+j) \pmod{p}.
\end{equation*}
Combining Proposition \ref{prop_pGammaG} (1) and (\ref{for_pZp}) yields the result.
\end{proof}

\begin{lemma}\label{lem_SumGammaG2}
Let $p\geq 7$ be a prime and let $x \in \mathbb{Z}_p$. If $0 \leq j \leq p-1\in \mathbb{Z}$, then
\begin{equation*}
G_1\left(x+j\right)^2-G_2\left(x+j\right) -G_1(1+j)^2+G_2(1+j)
\equiv
H_{\rep(x)-1+j}^{(2)}- H_{j}^{(2)}-\delta\\[3pt]
\pmod{p}
\end{equation*}
where 
\begin{equation*}
\delta=
\begin{cases}
0 & \text{if } 0\leq j\leq p- \rep(x),\\[3pt]
\frac{1}{p^2} & \text{if } p- \rep(x)+1\leq j\leq p-1.\\[3pt]
\end{cases}
\end{equation*}
\end{lemma}

\begin{proof}
By Corollary \ref{cor_pGammaCong1} (3) and (4) we see that
\begin{multline*}
G_1\left(x+j\right)^2-G_2\left(x+j\right) -G_1(1+j)^2+G_2(1+j)\\
\equiv
G_1\left(\rep(x)+j\right)^2-G_2\left(\rep(x)+j\right) -G_1(1+j)^2+G_2(1+j)
\pmod{p}.
\end{multline*}
Combining Proposition \ref{prop_pGammaG} (2) and (\ref{for_pZp}) yields the result.
\end{proof}

\begin{lemma}\label{lem_ProdGammap2}
Let $p\geq 7$ be a prime and let $x \in \mathbb{Z}_p$. Choose $m_1 \in \{x, 1-x\}$ such that $\rep(m_1) = \max\left(\rep(x) ,\rep(1-x)\right)$ and set $m_2=1-m_1$. Then for $0 \leq j<\rep(m_1) \in \mathbb{Z}$,
\begin{multline*}
\frac{\gfp{x+j}\gfp{1-x+j}}{\gfp{x}\gfp{1-x}{j!}^2} \equiv (-1)^{j} \bin{\rep(m_1)-1+j}{j} \bin{\rep(m_1)-1}{j}
\cdot \alpha\\
\cdot \left[1-\Bigl(\rep(m_1)-m_1\Bigr)\left(H_{\rep(m_1)-1+j}^{(1)}-H_{\rep(m_2)-1+j}^{(1)}-\beta \right)\right]
\pmod{p^2}
\end{multline*}
where 
\begin{align*}
\alpha=
\begin{cases}
1 & \text{if } 0\leq j \leq \rep(m_2)-1,\\
\frac{1}{p} & \text{if }\rep(m_2) \leq j < \rep(m_1),
\end{cases}
& \quad and \quad
\beta=
\begin{cases}
0 & \text{if } 0\leq j \leq \rep(m_2)-1,\\
\frac{1}{p} & \text{if }\rep(m_2) \leq j < \rep(m_1).
\end{cases}
\end{align*}
\end{lemma}

\begin{proof}
We first note that, by Proposition~\ref{prop_repOneminus},
\begin{equation}\label{for_mRels}
\rep(m_2)-p-m_2=-\left(\rep(m_1)-m_1\right) \quad \textup{and} \quad \rep(m_2)-p=1-\rep(m_1).
\end{equation}
Then by Proposition \ref{prop_pGammaCong} (2) we get that
\begin{align*}
&\gfp{x+j}\gfp{1-x+j}
=\gfp{m_1+j}\gfp{m_2+j}\\[6pt]
& \equiv
\left[\gfp{\rep(m_1)+j}-\left(\rep(m_1)-m_1\right)
\Gamma_p^{\prime}\left(\rep(m_1)+j\right) \right]\\[6pt]
& \quad \; \cdot
\left[\gfp{\rep(m_2)-p+j}-\left(\rep(m_2)-p-m_2\right)
\Gamma_p^{\prime}\left(\rep(m_2)+j\right) \right]\\[6pt]
& \equiv \gfp{\rep(m_1)+j} \gfp{1-\rep(m_1)+j} -
\left(\rep(m_1)-m_1\right)\\[6pt]
& \quad \; \cdot
\left[\Gamma_p^{\prime}\left(\rep(m_1)+j\right) \gfp{\rep(m_2)-p+j}
-\Gamma_p^{\prime}\left(\rep(m_2)+j\right) \gfp{\rep(m_1)+j} \right]
\pmod{p^2}.
\end{align*}
Using Proposition \ref{prop_pGamma} (3) we have
\begin{align*}
\Gamma_p^{\prime}&\left(\rep(m_1)+j\right) \gfp{\rep(m_2)-p+j}
-\Gamma_p^{\prime}\left(\rep(m_2)+j\right) \gfp{\rep(m_1)+j}\\[6pt]
&\equiv
\gfp{\rep(m_1)+j} \gfp{\rep(m_2)+j}
\left[G_1\left(\rep(m_1)+j\right)  -G_1\left(\rep(m_2)+j\right)\right]\\[6pt]
&\equiv
\gfp{\rep(m_1)+j} \gfp{1-\rep(m_1)+j}
\left[G_1\left(\rep(m_1)+j\right)  
-G_1\left(\rep(m_2)+j\right)\right]
\pmod{p}.
\end{align*}
So
\begin{multline*}
\gfp{x+j}\gfp{1-x+j}
\equiv
\gfp{\rep(m_1)+j} \gfp{1-\rep(m_1)+j} \\[6pt] \cdot
\left[1-\left(\rep(m_1)-m_1\right)
%\right. \\[3pt] \left.
\left(G_1\left(\rep(m_1)+j\right)  -G_1\left(\rep(m_2)+j\right)\right)\right]
\pmod{p^2}.
\end{multline*}
By Proposition \ref{prop_pGammaG} (1),
\begin{equation*}
G_1\left(\rep(m_1)+j\right)  -G_1\left(\rep(m_2)+j\right)
= H_{\rep(m_1)-1+j}^{(1)}-H_{\rep(m_2)-1+j}^{(1)}-\beta.
\end{equation*}
For $j<\rep(m_1)$, Proposition \ref{prop_pGamma} gives us
\begin{align*}
\gfp{\rep(m_1)+j} \gfp{1-\rep(m_1)+j}
&= (-1)^{\rep(m_1)-j} \frac{\gfp{\rep(m_1)+j}}{\gfp{\rep(m_1)-j}}\\[3pt]
&=(-1)^{\rep(m_1)-j} \frac{\left(\rep(m_1)-1+j\right)! \, (\alpha)}{\left(\rep(m_1)-1-j\right)! } \ffs
\end{align*}
The result follows from noting that
\begin{equation*}
\frac{\left(\rep(m_1)-1+j\right)!}{\left(\rep(m_1)-1-j\right)! \hspace{2pt} {j!}^2}
= \bin{\rep(m_1)-1+j}{j} \bin{\rep(m_1)-1}{j} \;
\end{equation*}
and
\begin{equation*}
\gfp{x}\gfp{1-x}=(-1)^{\rep(m_1)}.
\end{equation*}
\end{proof}

\begin{lemma}\label{lem_SumGammap2}
Let $p\geq 7$ be a prime and let $x \in \mathbb{Z}_p$. Choose $m_1 \in \{x, 1-x\}$ such that $\rep(m_1) = \max\left(\rep(x) ,\rep(1-x)\right)$ and set $m_2=1-m_1$. Then for $0 \leq j<\rep(m_1) \in \mathbb{Z}$,\begin{multline*}
G_1\left(x+j\right)+G_1\left(1-x+j\right) -2G_1(1+j)
\equiv
H_{\rep(m_1)-1+j}^{(1)}+H_{\rep(m_1)-1-j}^{(1)}-2\hspace{1pt} H_{j}^{(1)}-\alpha\\[3pt]
+\left(\rep(m_1)-m_1\right) 
\left(H_{\rep(m_1)-1+j}^{(2)}-H_{\rep(m_2)-1+j}^{(2)}-\beta \right)
\pmod{p^2}
\end{multline*}
where 
\begin{align*}
\alpha=
\begin{cases}
0 & \text{if }0\leq j \leq \rep(m_2)-1,\\
\frac{1}{p} &\text{if } \rep(m_2) \leq j < \rep(m_1),
\end{cases}
& \qquad and \qquad
\beta=
\begin{cases}
0 &\text{if } 0\leq j \leq \rep(m_2)-1,\\
\frac{1}{p^2} &\text{if } \rep(m_2) \leq j < \rep(m_1).
\end{cases}
\end{align*}\\[-12pt]
\end{lemma}

\begin{proof}
Using Corollary \ref{cor_pGammaCong1} and Corollary \ref{cor_pGammaCong2} we see that
\begin{multline*}
G_1\left(m_1+j\right)\equiv G_1\left(\rep(m_1)+j\right)\\+\left(\rep(m_1)-m_1\right)
\left(G_1\bigl(\rep(m_1)+j\bigr)^2-G_2\left(\rep(m_1)+j\right)\right)
\pmod{p^2}
\end{multline*}
and
\begin{multline*}
G_1\left(m_2+j\right)
\equiv G_1\left(\rep(m_2)-p+j\right)\\
+\left(\rep(m_2)-p-m_2\right)
\left(G_1\bigl(\rep(m_2)+j\bigr)^2-G_2\left(\rep(m_2)+j\right)\right)
\pmod{p^2}.
\end{multline*}

\noindent Therefore, using Proposition \ref{prop_pGammaG} (3) and (\ref{for_mRels}),
\begin{align*}
G_1\left(x+j\right)&+G_1\left(1-x+j\right)
\\[4pt]&\;
=G_1\left(m_1+j\right)+G_1\left(m_2+j\right)\\[4pt]
& \; \equiv
G_1\left(\rep(m_1)+j\right)+G_1\left(\rep(m_1)-j\right)+
\left(\rep(m_1)-m_1\right) \left[G_1\bigl(\rep(m_1)+j\bigr)^2
\right. \\[4pt] & \left. \qquad
-G_2\left(\rep(m_1)+j\right)
-G_1\bigl(\rep(m_2)+j\bigr)^2+G_2\bigl(\rep(m_2)+j\bigr)\right]
\pmod{p^2}.
\end{align*}

\noindent By Proposition \ref{prop_pGammaG} (1), (2) we get that
\begin{align*}
G_1\left(\rep(m_1)+j\right)-G_1(1+j)
=
H_{\rep(m_1)-1+j}^{(1)}- H_{j}^{(1)}-
\alpha,
\end{align*}
\begin{multline*}
G_1\bigl(\rep(m_1)+j\bigr)^2-G_2\left(\rep(m_1)+j\right)
-G_1\bigl(\rep(m_2)+j\bigr)^2+G_2\left(\rep(m_2)+j\right)\\[3pt]
=
H_{\rep(m_1)-1+j}^{(2)}-H_{\rep(m_2)-1+j}^{(2)}-\beta
\end{multline*}
and, for $0\leq j < \rep(m_1)$,
\begin{align*}
G_1\left(\rep(m_1)-j\right)-G_1(1+j)
=
H_{\rep(m_1)-1-j}^{(1)}- H_{j}^{(1)}.
\end{align*}
The result follows.
\end{proof}

\subsection{A Combinatorial Technique}
In this section we generalize a combinatorial technique from \cite{M} to produce Lemmas \ref{lem_P} and \ref{lem_Q} below. We will use both of these lemmas in proving Theorems \ref{thm_2G} to \ref{thm_4G2} in Section 4.
\begin{lemma}\label{lem_P}
Let $p$ be an odd prime and let $a_1, a_2, \cdots, a_n \in \mathbb{Z}^{+}$ be such that ${T:=\sum_{i=1}^{n} a_i \leq 2(p-1)}$. Then  
\begin{equation*}
\sum_{j=0}^{p-1}\left[ \prod_{i=1}^{n} \ph{j+1}{a_i}\right] \left[1+j\hspace{1pt} \sum_{i=1}^{n} \left( H_{a_i+j}^{(1)} - \hspace{1pt} H_{j}^{(1)}\right)\right] \equiv
\begin{cases}
0 & \textup{if } T< 2(p-1),\\
1 & \textup{if } T=2(p-1),
\end{cases}
\pmod{p}.
\end{equation*}
\end{lemma}

\begin{proof}
Let
\begin{equation*}
P(j):= \frac{d}{dj} \left[ j \prod_{i=1}^{n} \ph{j+1}{a_i} \right] = \sum_{k=0}^{T} b_k j^k.
\end{equation*}
Differentiating we see that
\begin{align}\label{for_DiffP}
\frac{d}{dj} \left[ j \prod_{i=1}^{n} \ph{j+1}{a_i} \right] 
\notag&= \left[ \prod_{i=1}^{n} \ph{j+1}{a_i} \right] \left[ 1 + j \sum_{i=1}^{n} \left( H_{a_i+j}^{(1)} - H_{j}^{(1)} \right)\right]\ffs 
\end{align}
So it suffices to show
\begin{equation*}
 \sum_{j=0}^{p-1} P(j) \equiv 
\begin{cases}
0 & \textup{if } T< 2(p-1),\\
1 & \textup{if } T=2(p-1),
\end{cases}
 \pmod{p}.
\end{equation*}
For a positive integer $k$, we have
\begin{equation}\label{exp_sums}
\sum^{p-1}_{j=1} j^k\equiv
\begin{cases}
-1 \pmod {p}& \text{if $(p-1) \vert k$} \, ,\\
\phantom{-}0 \pmod {p}& \text{otherwise} \, .
\end{cases}
\end{equation}
Consequently,
\begin{equation*}
\sum_{j=0}^{p-1} P(j)
\equiv \sum_{k=1}^{T} b_k \sum_{j=1}^{p-1}  j^k
\equiv
\begin{cases}
0 & \textup{if } T< p-1,\\
-b_{p-1} & \textup{if } p-1 \leq T< 2(p-1),\\
-b_{p-1} -b_{2(p-1)} & \textup{if }  T=2(p-1),
\end{cases}
\pmod{p}.
\end{equation*}
By definition of $P(j)$ we see that
\begin{equation*}
\prod_{i=1}^{n} \ph{j+1}{a_i} = \sum_{k=0}^{T} \frac{b_k}{k+1} j^k \;,
\end{equation*}
which is monic with integer coefficients. Thus  $p\mid b_{p-1}$ for $T \geq p-1$ and $b_{2(p-1)}=2p-1$ if $T=2(p-1)$. Therefore $b_{p-1}\equiv 0 \pmod{p}$ and $b_{2(p-1)}\equiv -1\pmod{p}$ in these cases.
\end{proof}

\begin{lemma}\label{lem_Q}
Let $p$ be an odd prime and let $a_1, a_2, \cdots, a_n \in \mathbb{Z}^{+}$ be such that ${T:=\sum_{i=1}^{n} a_i \leq 2(p-1)}$. Then  
\begin{multline*}
\sum_{j=0}^{p-1} \Biggl[ \prod_{i=1}^{n} \ph{j+1}{a_i} \Biggr] \Biggl[j\hspace{1pt} \sum_{i=1}^{n} \left( H_{a_i+j}^{(1)} - \hspace{1pt} H_{j}^{(1)}\right) 
+\frac{j^2}{2} \Biggl\{ 
  \left(\sum_{i=1}^{n} \left( H_{a_i+j}^{(1)} - \hspace{1pt} H_{j}^{(1)}\right)\right)^2
  \\
- \sum_{i=1}^{n} \left( H_{a_i+j}^{(2)} - \hspace{1pt} H_{j}^{(2)}\right) \Biggr\}
 \Biggr]  \equiv 
\begin{cases}
0 & \textup{if } T< 2(p-1),\\
-1 & \textup{if } T=2(p-1),
\end{cases}
\pmod{p}.
\end{multline*}
\end{lemma}

\begin{proof}
We first recognize that
\begin{multline*}
\Biggl[\prod_{i=1}^{n} \ph{j+1}{a_i} \Biggr] \Biggl[j\hspace{1pt} \sum_{i=1}^{n} \left( H_{a_i+j}^{(1)} - \hspace{1pt} H_{j}^{(1)}\right) 
+\frac{j^2}{2} \Biggl\{\left(\sum_{i=1}^{n} \left( H_{a_i+j}^{(1)} - \hspace{1pt} H_{j}^{(1)}\right)\right)^2
- \sum_{i=1}^{n} \left( H_{a_i+j}^{(2)} - \hspace{1pt} H_{j}^{(2)}\right) \Biggr\} \Biggr]
\\
=\frac{j}{2} \frac{d^2}{dj^2} \left[ j \prod_{i=1}^{n} \ph{j+1}{a_i} \right]
= \sum_{k=0}^{T} b_k j^k. 
\end{multline*}
Then applying (\ref{exp_sums}) in a similar fashion to that used in the proof of Lemma \ref{lem_P} yields the result.
\end{proof}

\subsection{Binomial Coefficient-Generalized Harmonic Sum Identities}
We now state two binomial coefficient--generalized harmonic sum identities from \cite{DMC3} which we will use in Section 4.
\begin{theorem}[\cite{DMC3} Thm. 2]\label{Cor_BinHarId1}
Let $m,n$ be positive integers with $m\geq n$. Then
\begin{multline*}
\sum_{k=0}^{n} \biggbin{m+k}{k} \biggbin{m}{k} \biggbin{n+k}{k} \biggbin{n}{k}
 \biggl[ 1+k \left(H_{m+k}^{(1)} +H_{m-k}^{(1)} + H_{n+k}^{(1)} 
% \right. \\[6pt] \left. 
 + H_{n-k}^{(1)} -4H_k^{(1)}\right) \biggr]\\[6pt]
\notag  +\sum_{k=n+1}^{m} (-1)^{k-n} \biggbin{m+k}{k} \biggbin{m}{k} \biggbin{n+k}{k} \Big/ \biggbin{k-1}{n}
=(-1)^{m+n}.
\end{multline*}
\end{theorem}
\begin{theorem}[\cite{DMC3} Thm. 3]\label{Cor_BinHarId2}
Let $l,m, n$ be positive integers with $l > m\geq n\geq\frac{l}{2}$ and $c_1, c_2 \in \mathbb{Q}$ some constants. Then
\begin{multline*}
\sum_{k=0}^{n} \biggbin{m+k}{k} \biggbin{m}{k} \biggbin{n+k}{k} \biggbin{n}{k} 
 \Biggl\{ \biggl[1+k \Bigl(H_{m+k}^{(1)} +H_{m-k}^{(1)} + H_{n+k}^{(1)} + H_{n-k}^{(1)} 
-4H_k^{(1)} \Bigr)\biggr]
\\[5pt]
 \cdot \biggl[c_1\left(H_{k+n}^{(1)} - H_{k+l-n-1}^{(1)}\right)  + c_2 \Bigl(H_{k+m}^{(1)} - H_{k+l-m-1}^{(1)}\Bigr)\biggr] 
-k\biggl[c_1\left(H_{k+n}^{(2)} - H_{k+l-n-1}^{(2)}\right)\\[5pt] 
+ c_2 \left(H_{k+m}^{(2)} - H_{k+l-m-1}^{(2)}\right)\biggr] \Biggr\}
+ \sum_{k=n+1}^{m} (-1)^{k-n} \biggbin{m+k}{k} \biggbin{m}{k} \biggbin{n+k}{k} \Big/ \biggbin{k-1}{n} 
\\[5pt]
\cdot \biggl[c_1\left(H_{k+n}^{(1)} - H_{k+l-n-1}^{(1)}\right) + c_2 \left(H_{k+m}^{(1)} - H_{k+l-m-1}^{(1)}\right)\biggr] = 0.
\end{multline*}
\end{theorem}

%%%%%%%%%%%%%%%%%%%%%%%%%%%%%%%%%%%%%%%%%%%%%%%%%%%
%%%%%%%%%%%%%%%%%%%%%%%%%%%%%%%%%%%%%%%%%%%%%%%%%%%
%%%%%%%%%%%%%%%%%%%%%%%%%%%%%%%%%%%%%%%%%%%%%%%%%%%
%%%%%%%%%%%%%%%%%%%%%%%%%%%%%%%%%%%%%%%%%%%%%%%%%%%

\section{Proofs} 
The proofs of Theorems \ref{thm_2G} to \ref{thm_4G2} proceed broadly along similar lines with the overall idea being to expand the relevant $_{n+1}G$ function using properties of the $p$-adic gamma function and its logarithmic derivatives,
 identify the contribution of the truncated generalized hypergeometric series, and then use the results of Section 3.2 and 3.3 to simplify the remaining terms. 
In doing this we extend significantly the methods in \cite{AO}, \cite{K} and \cite{M}. For brevity we include only the proof of Theorem \ref{thm_4G2} which is the most complex and uses all the additional methods developed for these results. Full details of all proofs can be found in \cite{DMC}.

\begin{proof}[Proof of Theorem \ref{thm_4G2}] 
Assume without loss of generality that $r<d/2$. We easily check the result for primes $p<7$. 
%There are six distinct $(p,d,r)$ triples in this category, namely: $(3,5,2)$, $(3,8,3)$, $(3, 10, 3)$, $(5,8,3)$, $(5, 12, 5)$ and $(5,13, 5)$. 
Let $p\geq 7$ be a prime which satisfies the conditions of the theorem.
Let $p\equiv a \pmod{d}$ with $0<a<d$. Then $a \in \{1,r,d-r,d-1\}$.
We define $s:=ar-d \left\lfloor \frac{ar}{d} \right\rfloor = d \, \langle \frac{ar}{d} \rangle$. 
It is easy to check that $\{1, r, d-r, d-1\} = \{a, s, d-s, d-a\}$ for all possible $a$.

From Lemma \ref{cor_repGenformula} we know that $\rep(\frac{a}{d}) = p- \lfloor \frac{p-1}{d} \rfloor$ and $\rep(\frac{d-a}{d}) = \lfloor \frac{p-1}{d}\rfloor+1$.
It is easy to check that $\rep(\frac{s}{d})=p-r \lfloor \tfrac{p-1}{d} \rfloor - \left\lfloor \tfrac{ar}{d} \right\rfloor$. Then, by Proposition \ref{prop_repOneminus}, $\rep(\frac{d-s}{d})= r \lfloor \tfrac{p-1}{d} \rfloor +1 +\left\lfloor \tfrac{ar}{d} \right\rfloor$.
Therefore $\Bigl\{\rep\left(\tfrac{1}{d}\right),\rep\left(\tfrac{r}{d}\right),\rep\left(\tfrac{d-r}{d}\right),\rep\left(\tfrac{d-1}{d}\right)\Bigr\}=\Bigl\{\lfloor \tfrac{p-1}{d} \rfloor +1,r \lfloor \tfrac{p-1}{d}\rfloor +1 +\lfloor \tfrac{ar}{d} \rfloor, p-r \lfloor \tfrac{p-1}{d} \rfloor - \lfloor \tfrac{ar}{d} \rfloor ,p-\lfloor \tfrac{p-1}{d} \rfloor\Bigr\}$, where the exact correspondence between the elements of each set depends on the choice of $p$. If we let $m_1:=\frac{a}{d}$, $m_2:=\frac{s}{d}$, $m_3:=\frac{d-s}{d}$ and $m_4:=\frac{d-a}{d}$ then $\rep\bigl(m_4\bigr)< \rep\bigl(m_3\bigr) \leq \rep\bigl(m_2\bigr) < \rep\bigl(m_1\bigr)$.

We reduce Definition \ref{def_GFn} modulo $p^3$ while using Proposition \ref{prop_pGamma} (2), (3) to expand the terms involved, noting that $t:=1+p+p^2 \equiv \frac{1}{1-p} \pmod{p^3}$, to get
\begin{align*}
&{_{4}G}  \left(\tfrac{1}{d} , \tfrac{r}{d}, 1-\tfrac{r}{d} , 1-\tfrac{1}{d}\right)_p
\equiv
\sum_{j=0}^{\left\lfloor \frac{p-1}{d} \right\rfloor} 
\frac{\biggfp{\frac{1}{d}+jt}\biggfp{\frac{r}{d}+jt}
\biggfp{\frac{d-r}{d}+jt}\biggfp{\frac{d-1}{d}+jt}}
{\biggfp{\frac{1}{d}}\biggfp{\frac{r}{d}}\biggfp{\frac{d-r}{d}}\biggfp{\frac{d-1}{d}}
{\biggfp{1+jt}}^{4}} 
\\ & 
+p \left\{
\sum_{j=0}^{\left\lfloor \frac{p-1}{d} \right\rfloor} 
\frac{\biggfp{\frac{1}{d}+j+jp}\biggfp{\frac{r}{d}+j+jp}
\biggfp{\frac{d-r}{d}+j+jp}\biggfp{\frac{d-1}{d}+j+jp}}
{\biggfp{\frac{1}{d}}\biggfp{\frac{r}{d}}\biggfp{\frac{d-r}{d}}\biggfp{\frac{d-1}{d}}
{\biggfp{1+j+jp}}^{4}} 
\right. \\ &\left. \qquad
- \sum_{j=\left\lfloor \frac{p-1}{d} \right\rfloor+1}^{\left\lfloor \frac{r(p-1)}{d} \right\rfloor}
\frac{\biggfp{\frac{d+1}{d}+j+jp}\biggfp{\frac{r}{d}+j+jp}
\biggfp{\frac{d-r}{d}+j+jp}\biggfp{\frac{d-1}{d}+j+jp}}
{\biggfp{\frac{1}{d}}\biggfp{\frac{r}{d}}\biggfp{\frac{d-r}{d}}\biggfp{\frac{d-1}{d}}
{\biggfp{1+j+jp}}^{4}}
\right\} \\ & 
+p^2 \left\{
\sum_{j=0}^{\left\lfloor \frac{p-1}{d} \right\rfloor} 
\frac{\biggfp{\frac{1}{d}+j}\biggfp{\frac{r}{d}+j}
\biggfp{\frac{d-r}{d}+j}\biggfp{\frac{d-1}{d}+j}}
{\biggfp{\frac{1}{d}}\biggfp{\frac{r}{d}}\biggfp{\frac{d-r}{d}}\biggfp{\frac{d-1}{d}}
{\biggfp{1+j}}^{4}} 
\right. \\ &\left. \qquad \; \;
 - \sum_{j=\left\lfloor \frac{p-1}{d} \right\rfloor+1}^{\left\lfloor \frac{r(p-1)}{d} \right\rfloor}
\frac{\biggfp{\frac{d+1}{d}+j}\biggfp{\frac{r}{d}+j}
\biggfp{\frac{d-r}{d}+j}\biggfp{\frac{d-1}{d}+j}}
{\biggfp{\frac{1}{d}}\biggfp{\frac{r}{d}}\biggfp{\frac{d-r}{d}}\biggfp{\frac{d-1}{d}}
{\biggfp{1+j}}^{4}}
\right. \\ &\left.  \qquad \; \;
+\sum_{j=\left\lfloor \frac{r(p-1)}{d} \right\rfloor+1}^{\left\lfloor (d-r)\frac{p-1}{d} \right\rfloor} 
\frac{\biggfp{\frac{d+1}{d}+j}\biggfp{\frac{d+r}{d}+j}
\biggfp{\frac{d-r}{d}+j}\biggfp{\frac{d-1}{d}+j}}
{\biggfp{\frac{1}{d}}\biggfp{\frac{r}{d}}\biggfp{\frac{d-r}{d}}\biggfp{\frac{d-1}{d}}
{\biggfp{1+j}}^{4}}
\right\} 
\pmod{p^3}.
\end{align*}

Central to the proof will be the relationship between the $\rep(m_k)$, for $2 \leq k \leq 4$, and the limits of summation of the individual sums in the expanded $_4G$ above. These relationships are outlined below and the reader may want to refer to them throughout the rest of the proof.\\
\begin{itemize}
\item 
$\rep(m_4)= \lfloor \tfrac{p-1}{d} \rfloor + 1$ .\\[9pt] 
\item
$\rep(m_3)= \bigl\lfloor \frac{r(p-1)}{d} \bigr\rfloor + 1 +
\begin{cases}
1 & \textup{if } p \equiv r \pmod{d} \textup{ and } r^2 \equiv 1 \pmod{d},\\[3pt]
1 & \textup{if } p \equiv d-r \pmod{d} \textup{ and } r^2 \equiv -1 \pmod{d},\\[3pt]
0 & \textup{otherwise.}
\end{cases}
$  \\[12pt]
\item
$\rep(m_2) = \bigl\lfloor (d-r)\frac{(p-1)}{d} \bigr\rfloor + 1 +
\begin{cases}
1 & \textup{if } p \equiv r \pmod{d} \textup{ and } r^2 \equiv -1 \pmod{d},\\[3pt]
1 & \textup{if } p \equiv d-r \pmod{d} \textup{ and } r^2 \equiv 1 \pmod{d},\\[3pt]
1 & \textup{if } p \equiv  d-1 \pmod{d},\\[3pt]
0 & \textup{otherwise.}
\end{cases}
$ \\[6pt] 
\end{itemize}

 \noindent We now consider $\gfp{\frac{d+1}{d}+j+jp}$ and $\gfp{\frac{d+1}{d}+j}$. 
As $\rep\bigl(\tfrac{d-1}{d}\bigr)=p+1-\rep\bigl(\tfrac{1}{d}\bigr)$ by Proposition \ref{prop_repOneminus}, we have that
\begin{equation*}
\tfrac{1}{d}+j +jp \in p\mathbb{Z}_p \Longleftrightarrow \tfrac{1}{d}+j \in p\mathbb{Z}_p 
\Longleftrightarrow \rep\left(\tfrac{1}{d}\right)+j \in p\mathbb{Z}_p
\Longleftrightarrow \rep\left(\tfrac{1}{d}\right)+j =p
\Longleftrightarrow j= \rep\left(\tfrac{d-1}{d}\right) -1.
\end{equation*}
Consequently, we see that the only time that $\tfrac{1}{d}+j \in p\mathbb{Z}_p$ for $\lfloor \tfrac{p-1}{d} \rfloor +1 \leq j \leq \lfloor (d-r)\tfrac{p-1}{d}\rfloor$ is when $p\equiv r \pmod{d}$ with $r^2\equiv 1 \pmod{d}$ or $p\equiv d-r \pmod{d}$ with $r^2\equiv -1 \pmod{d}$ and in these cases $j=\rep(m_3)-1=\bigl\lfloor  \tfrac{r(p-1)}{d} \bigr\rfloor +1$.
\noindent Therefore, using Proposition \ref{prop_pGamma} (1) we get that for $\lfloor \tfrac{p-1}{d}\rfloor +1 \leq j \leq \lfloor (d-r)\tfrac{p-1}{d} \rfloor$,
\begin{equation*}
\gfp{\tfrac{d+1}{d}+j}=
\begin{cases}
-\gfp{\tfrac{1}{d}+j} & \text{if } j=\bigl\lfloor \tfrac{r(p-1)}{d} \bigr\rfloor +1,\; p\equiv r \imod{d}, \;r^2\equiv 1 \imod{d},\\[3pt]
-\gfp{\tfrac{1}{d}+j} & \text{if }j=\bigl\lfloor \tfrac{r(p-1)}{d} \bigr\rfloor +1, \;p\equiv d-r \imod{d}, \; r^2\equiv -1 \imod{d},\\[3pt]
-\gfp{\tfrac{1}{d}+j} \left(\tfrac{1}{d}+j\right) & \textup{otherwise,}
\end{cases}
\end{equation*}
and that 
$$\gfp{\tfrac{d+1}{d}+j+jp}=-\left(\tfrac{1}{d}+j+jp\right)\gfp{\tfrac{1}{d}+j+jp}$$
for $\lfloor \tfrac{p-1}{d} \rfloor +1 \leq j \leq \bigl\lfloor \tfrac{r(p-1)}{d} \bigr\rfloor$.
Considering $\gfp{\frac{d+r}{d}+j}$ in a similar fashion, we get that
$$\gfp{\tfrac{d+r}{d}+j}=-\left(\tfrac{r}{d}+j\right)\gfp{\tfrac{r}{d}+j}$$
for $\bigl\lfloor \frac{r(p-1)}{d} \bigr\rfloor+1 \leq j \leq  {\lfloor (d-r)\frac{p-1}{d} \rfloor}$. 
Applying these results and substituting $\{m_k\}$ for $\{\frac{1}{d}, \frac{r}{d}, \frac{d-r}{d}, \frac{d-1}{d}\}$ yields 
\begin{align*}
&{_{4}G} \left(\tfrac{1}{d} , \tfrac{r}{d}, 1-\tfrac{r}{d} , 1-\tfrac{1}{d}\right)_p
\equiv
\sum_{j=0}^{\left\lfloor \frac{p-1}{d} \right\rfloor} 
\left[ \, \prod_{k=1}^{4}  \frac{\biggfp{m_k+j+jp+jp^2}}
{\biggfp{m_k}
{\biggfp{1+j+jp+jp^2}}}\right] 
\\  &
 +p \left\{
\sum_{j=0}^{\left\lfloor \frac{p-1}{d} \right\rfloor} 
\left[ \, \prod_{k=1}^{4}  \frac{ \biggfp{m_k+j+jp}}
{ \biggfp{m_k}
{\biggfp{1+j+jp}}}\right] 
 + \sum_{j=\left\lfloor \frac{p-1}{d} \right\rfloor+1}^{\left\lfloor \frac{r(p-1)}{d} \right\rfloor}
\left[ \, \prod_{k=1}^{4} \frac{ \biggfp{m_k+j+jp} }
{ \biggfp{m_k}
{\biggfp{1+j+jp}}}\right] 
 \bigl(\tfrac{1}{d}+j+jp\bigr)
\right\} \\ &
+p^2 \left\{
\sum_{j=0}^{\left\lfloor \frac{p-1}{d} \right\rfloor} 
\left[ \,\prod_{k=1}^{4} \frac{ \biggfp{m_k+j}}
{ \biggfp{m_k}
{\biggfp{1+j}}} \right] 
+ \sum_{j=\left\lfloor \frac{p-1}{d} \right\rfloor+1}^{\left\lfloor \frac{r(p-1)}{d} \right\rfloor}
\left[\,\prod_{k=1}^{4} \frac{ \biggfp{m_k+j}}
{ \biggfp{m_k}
{\biggfp{1+j}}} \right] 
 \bigl(\tfrac{1}{d}+j\bigr)
\right. \\ & \left. \qquad \quad \;
+\sum_{j=\left\lfloor \frac{r(p-1)}{d} \right\rfloor+1}^{\rep(m_3)-1} 
\left[\, \prod_{k=1}^{4} \frac{ \biggfp{m_k+j} }
{\biggfp{m_k}
{\biggfp{1+j}}} \right]
\bigl(\tfrac{r}{d}+j\bigr)
\right. \\ & \left. \qquad \quad \;
+\sum_{j=\rep(m_3)}^{\left\lfloor (d-r)\frac{p-1}{d} \right\rfloor} 
\left[\, \prod_{k=1}^{4} \frac{ \biggfp{m_k+j} }
{ \biggfp{m_k}
{\biggfp{1+j}}} \right] 
\bigl(\tfrac{1}{d}+j\bigr)\bigl(\tfrac{r}{d}+j\bigr)
\right\} 
\pmod{p^3},
\end{align*}
where the second to last sum is vacuous unless $p\equiv r \pmod{d}$ and $r^2\equiv 1 \pmod{d}$ or $p\equiv d-r \pmod{d}$ and $r^2\equiv -1 \pmod{d}$. 
By Proposition \ref{prop_pGammaCong} (2) we see that, for $1 \leq k \leq 4$,
\begin{equation*}
\biggfp{m_k+j+jp+jp^2}\\
\equiv
\biggfp{m_k+j}
\left[1+(jp+jp^2)\,G_1(m_k+j)+\tfrac{j^2p^2}{2}\,G_2(m_k+j)\right]  
\pmod{p^3},
\end{equation*}
and also
\begin{equation*}
\biggfp{1+j+jp+jp^2}^{4}
\equiv \biggfp{1+j}^4
\left[1+(jp+jp^2)\, G_1(1+j)
+\tfrac{j^2p^2}{2}\,G_2(1+j)\right]^4
\pmod{p^3}.
\end{equation*}

\noindent Multiplying the numerator and denominator by 
$${1-4(jp+jp^2)\;G_1(1+j)-2j^2p^2\, \left(G_2(1+j)-5\,G_1(1+j)^2\right)}$$ we get that
\begin{equation*}
\frac{  \displaystyle \prod_{k=1}^{4} \left[ 1+(jp+jp^2)\,G_1(m_k+j)+\tfrac{j^2p^2}{2}\,G_2(m_k+j)\right] }
{\left[1+(jp+jp^2)\,G_1(1+j)
+\frac{j^2p^2}{2}\,G_2(1+j)\right]^4}
\equiv
1+(jp+jp^2)A(j)+j^2p^2B(j)
\imod{p^3},
\end{equation*}
where
$$A(j):=\sum_{k=1}^{4} \Bigl(G_1(m_k+j)-G_1(1+j)\Bigr)$$
and
$$B(j):=\frac{1}{2}\left[A(j)^2 - \sum_{k=1}^{4} \Bigl(G_1(m_k+j)^2-G_2(m_k+j)
-G_1(1+j)^2 +G_2(1+j)\Bigr)\right]\ffs$$\\

We note that both $A(j)$ and $B(j)$ $\in \mathbb{Z}_p$ by Proposition \ref{prop_pGammaCong} (1). Applying the above and noting that $\gfp{1+j}=(-1)^{1+j} j!$ for $j<p$, we get, after rearranging,
\begin{align}\label{for_4G2}
\notag {_{4}G}& \left(\tfrac{1}{d} , \tfrac{r}{d}, 1-\tfrac{r}{d} , 1-\tfrac{1}{d}\right)_p
\equiv
\sum_{j=0}^{\left\lfloor \frac{p-1}{d} \right\rfloor} 
\left[ \, \prod_{k=1}^{4} \frac{ \biggfp{m_k+j}}
{\biggfp{m_k}\, j!} \right]
\\[3pt] \notag & \: \,
+p \left\{
\sum_{j=0}^{\left\lfloor \frac{p-1}{d} \right\rfloor} 
\left[ \, \prod_{k=1}^{4} \frac{ \biggfp{m_k+j}}
{\biggfp{m_k}\, j!} \right]
\Bigl[1+jA(j)\Bigr]
+ \sum_{j=\left\lfloor \frac{p-1}{d} \right\rfloor+1}^{\left\lfloor \frac{r(p-1)}{d} \right\rfloor} 
\left[ \, \prod_{k=1}^{4} \frac{ \biggfp{m_k+j}}
{\biggfp{m_k}\, j!} \right]
\Bigl[\tfrac{1}{d}+j\Bigr]
\right\} 
\\[3pt] \notag &
 +p^2 \left\{
\sum_{j=0}^{\left\lfloor \frac{p-1}{d} \right\rfloor} 
\left[ \, \prod_{k=1}^{4} \frac{ \biggfp{m_k+j}}
{\biggfp{m_k}\, j!} \right]
\Bigl[1+2jA(j)+j^2B(j)\Bigr]
\right. \\[3pt] \notag & \left. \qquad \quad \;
+\sum_{j=\left\lfloor \frac{p-1}{d} \right\rfloor+1}^{\left\lfloor \frac{r(p-1)}{d} \right\rfloor} 
\left[ \, \prod_{k=1}^{4} \frac{ \biggfp{m_k+j}}
{\biggfp{m_k}\, j!} \right]
\Bigl[\Bigl(\tfrac{1}{d}+j\Bigr)\Bigl(1+jA(j)\Bigr)+j\Bigr]
\right. \\[3pt] \notag &\left.  \qquad \quad \;
+\sum_{j=\left\lfloor \frac{r(p-1)}{d} \right\rfloor+1}^{\rep(m_3)-1} 
\left[ \, \prod_{k=1}^{4} \frac{ \biggfp{m_k+j}}
{\biggfp{m_k}\, j!} \right]
\Bigl[\tfrac{r}{d}+j\Bigr]
\right. \\[3pt] &\left.  \qquad \quad \;
+\sum_{j=\rep(m_3)}^{\left\lfloor (d-r)\frac{p-1}{d} \right\rfloor} 
\left[ \, \prod_{k=1}^{4} \frac{ \biggfp{m_k+j}}
{\biggfp{m_k}\, j!} \right]
\Bigl[\Bigl(\tfrac{1}{d}+j\Bigr)\Bigl(\tfrac{r}{d}+j\Bigr)\Bigr]
\right\}
\pmod{p^3}.
\end{align}

\noindent Proposition \ref{prop_gammapmd} gives us
\begin{align}\label{for_4F32}
\notag {_{4}F_3} & \Biggl[ \begin{array}{cccc} \frac{1}{d}, & \frac{r}{d}, & 1-\frac{r}{d}, & 1-\frac{1}{d}\vspace{.05in}\\
\phantom{\frac{1}{d}} & 1, & 1, & 1 \end{array}
\bigg| \; 1 \Biggr]_{p-1} =\sum_{j=0}^{p-1} \,
\prod_{k=1}^{4} \frac{\ph{m_k}{j}}{{j!}}
\\ \notag &
\equiv
\sum_{j=0}^{\rep(m_4)-1}
\left[ \, \prod_{k=1}^{4} \frac{ \biggfp{m_k+j}}
{\biggfp{m_k}\, j!} \right]
%\\ \notag &
+\sum_{j=\rep(m_4)}^{\rep(m_3)-1} 
\left[ \, \prod_{k=1}^{4} \frac{ \biggfp{m_k+j}}
{\biggfp{m_k}\, j!} \right]
\bigl(m_1+p-\rep(m_1)\bigr)
\\ \notag & \qquad \: \,
 {+\sum_{j=\rep(m_3)}^{\rep(m_2) - 1} 
\left[ \, \prod_{k=1}^{4} \frac{ \biggfp{m_k+j}}
{\biggfp{m_k}\, j!} \right]
\bigl(m_1+p-\rep(m_1)) \bigl(m_2+p-\rep(m_2))}
\\ \notag &
 \equiv
\sum_{j=0}^{\rep(m_4)-1}
\left[ \, \prod_{k=1}^{4} \frac{ \biggfp{m_k+j}}
{\biggfp{m_k}\, j!} \right]
+p\sum_{j=\rep(m_4)}^{\rep(m_3)-1} 
\left[ \, \prod_{k=1}^{4} \frac{ \biggfp{m_k+j}}
{\biggfp{m_k}\, j!} \right]
\Bigl(\tfrac{1}{d}\Bigr)
\\& \qquad \:\,
+p^2\sum_{j=\rep(m_3)}^{\rep(m_2) - 1} 
\left[ \, \prod_{k=1}^{4} \frac{ \biggfp{m_k+j}}
{\biggfp{m_k}\, j!} \right]
\Bigl(\tfrac{1}{d}\Bigr) \Bigl(\tfrac{r}{d}\Bigr)
\pmod {p^3}.
\end{align}
This last step uses the fact that, by definition,
\begin{equation*}
m_1+p - \rep(m_1)=\tfrac{a}{d}+p-\left(p- \lfloor \tfrac{p-1}{d} \rfloor \right)
=\tfrac{a}{d}+ \tfrac{p-a}{d} = \tfrac{p}{d},\\[0pt]
\end{equation*}
and
\begin{equation*}
m_2+p - \rep(m_2)
=\tfrac{s}{d}+ \tfrac{r(p-a)}{d} +\left\lfloor \tfrac{ar}{d} \right\rfloor
=\tfrac{ar}{d} - \left\lfloor \tfrac{ar}{d} \right\rfloor + \tfrac{rp}{d} -\tfrac{ar}{d} + \left\lfloor \tfrac{ar}{d} \right\rfloor
= \tfrac{rp}{d}.
\end{equation*}

\noindent Therefore, combining (\ref{for_4G2}) and (\ref{for_4F32}), it suffices to prove
\begin{align}\label{Resid1_4G2}
\notag & \sum_{j=0}^{\rep(m_4)-1} 
\left[ \, \prod_{k=1}^{4} \frac{ \biggfp{m_k+j}}
{\biggfp{m_k}\, j!} \right]
\Bigl[1+jA(j)\Bigr]
+ \sum_{j=\rep(m_4)}^{\left\lfloor \frac{r(p-1)}{d} \right\rfloor} 
\left[ \, \prod_{k=1}^{4} \frac{ \biggfp{m_k+j}}
{\biggfp{m_k}\, j!} \right]
\Bigl[j\Bigr]
\\ \notag &
{- \sum_{j=\left\lfloor \frac{r(p-1)}{d} \right\rfloor+1}^{\rep(m_3)-1} 
\left[ \, \prod_{k=1}^{4} \frac{ \biggfp{m_k+j}}
{\biggfp{m_k}\, j!} \right]
\Bigl[\tfrac{1}{d}\Bigr]}
%\\ \notag &
 +p \left\{
\sum_{j=0}^{\rep(m_4)-1} 
\left[ \, \prod_{k=1}^{4} \frac{ \biggfp{m_k+j}}
{\biggfp{m_k}\, j!} \right]
\Bigl[1+2jA(j)+j^2B(j)\Bigr]
\right. \\ \notag & \left.
{+\sum_{j=\rep(m_4)}^{\left\lfloor \frac{r(p-1)}{d} \right\rfloor} 
\left[ \, \prod_{k=1}^{4} \frac{ \biggfp{m_k+j}}
{\biggfp{m_k}\, j!} \right]
\Bigl[\Bigl(\tfrac{1}{d}+j\Bigr)\Bigl(1+jA(j)\Bigr)+j\Bigr]}
\right. \\ \notag & \left. 
{+\sum_{j=\left\lfloor \frac{r(p-1)}{d} \right\rfloor+1}^{\rep(m_3)-1} 
\left[ \, \prod_{k=1}^{4} \frac{ \biggfp{m_k+j}}
{\biggfp{m_k}\, j!} \right]
\Bigl[\tfrac{r}{d}+j\Bigr]} 
%\right. \\ \notag & \left. 
{+\sum_{j=\rep(m_3)}^{\left\lfloor (d-r)\frac{p-1}{d} \right\rfloor} 
\left[ \, \prod_{k=1}^{4} \frac{ \biggfp{m_k+j}}
{\biggfp{m_k}\, j!} \right]
\Bigl[j^2+j\Bigl(\tfrac{1}{d}+\tfrac{r}{d}\Bigr)\Bigr]}
\right. \\ & \left. 
- \sum_{j=\left\lfloor (d-r)\frac{p-1}{d} \right\rfloor+1}^{\rep(m_2) - 1} 
\left[ \, \prod_{k=1}^{4} \frac{ \biggfp{m_k+j}}
{\biggfp{m_k}\, j!} \right]
\Bigl[\Bigl(\tfrac{1}{d}\Bigr) \Bigl(\tfrac{r}{d}\Bigr)\Bigr]
\right\} 
\equiv s(p)
\pmod{p^2}.
\end{align}

\noindent We note that the terms inside the braces in (\ref{Resid1_4G2}) need only be considered modulo $p$ and can be rewritten as follows.
\begin{align}\label{Resid1_4G2_split}
\notag &\sum_{j=0}^{\rep(m_4)-1} 
\left[ \, \prod_{k=1}^{4} \frac{ \biggfp{m_k+j}}
{\biggfp{m_k}\, j!} \right]
\Bigl[1+2jA(j)+j^2B(j)\Bigr]
+\sum_{j=\rep(m_4)}^{\left\lfloor \frac{r(p-1)}{d} \right\rfloor} 
\left[ \, \prod_{k=1}^{4} \frac{ \biggfp{m_k+j}}
{\biggfp{m_k}\, j!} \right]
\Bigl[\Bigl(\tfrac{1}{d}\Bigr)\Bigl(1+jA(j)\Bigr)\Bigr]
\\ &
\notag +\sum_{j=\left\lfloor \frac{r(p-1)}{d} \right\rfloor+1}^{\rep(m_3)-1} 
\left[ \, \prod_{k=1}^{4} \frac{ \biggfp{m_k+j}}
{\biggfp{m_k}\, j!} \right]
\Bigl[\tfrac{r}{d}-j-j^2A(j)\Bigr]
+\sum_{j=\rep(m_4)}^{\rep(m_3)-1} 
\left[ \, \prod_{k=1}^{4} \frac{ \biggfp{m_k+j}}
{\biggfp{m_k}\, j!} \right]
\Bigl[2j+j^2A(j)\Bigr]
\\ &
\notag +\sum_{j=\rep(m_3)}^{\left\lfloor (d-r)\frac{p-1}{d} \right\rfloor} 
\left[ \, \prod_{k=1}^{4} \frac{ \biggfp{m_k+j}}
{\biggfp{m_k}\, j!} \right]
\Bigl[j\Bigl(\tfrac{1}{d}+\tfrac{r}{d}\Bigr)\Bigr]
- \sum_{j=\left\lfloor (d-r)\frac{p-1}{d} \right\rfloor+1}^{\rep(m_2) - 1} 
\left[ \, \prod_{k=1}^{4} \frac{ \biggfp{m_k+j}}
{\biggfp{m_k}\, j!} \right]
\Bigl[\Bigl(\tfrac{1}{d}\Bigr) \Bigl(\tfrac{r}{d}\Bigr)+j^2\Bigr] 
\\ &
+\sum_{\rep(m_3)}^{\rep(m_2)-1} 
\left[ \, \prod_{k=1}^{4} \frac{ \biggfp{m_k+j}}
{\biggfp{m_k}\, j!} \right]
\Bigl[j^2\Bigr].
 \end{align}

\noindent We now consider the first, fourth and last terms of (\ref{Resid1_4G2_split}). Define 
\begin{multline}\label{for_Xj}
X(j):=\sum_{j=0}^{\rep(m_4)-1} 
\left[ \, \prod_{k=1}^{4} \frac{ \biggfp{m_k+j}}
{\biggfp{m_k}\, j!} \right]
\Bigl[1+2jA(j)+j^2B(j)\Bigr]
\\
+\sum_{j=\rep(m_4)}^{\rep(m_3)-1} 
\left[ \, \prod_{k=1}^{4} \frac{ \biggfp{m_k+j}}
{\biggfp{m_k}\, j!} \right]
\Bigl[2j+j^2A(j)\Bigr]
+\sum_{\rep(m_3)}^{\rep(m_2)-1} 
\left[ \, \prod_{k=1}^{4} \frac{ \biggfp{m_k+j}}
{\biggfp{m_k}\, j!} \right]
\Bigl[j^2\Bigr].
\end{multline}

\noindent We will show $X(j)\equiv0 \pmod{p}$. We start by examining $A(j)$, $B(j)$, $\prod_{k=1}^{4} \biggfp{m_k+j}$ and $\prod_{k=1}^{4} \biggfp{m_k}$ modulo $p$. Define, for $t\in \{1,2\}$,
\begin{align*} 
\delta_t:=
\begin{cases}
0 & \text{if } 0\leq j\leq \rep(m_4)-1, \\[6pt]
\frac{1}{p^t} &\text{if }  \rep(m_4) \leq j\leq \rep(m_3)-1,\\[6pt]
\frac{2}{p^t} & \text{if }  \rep(m_3)\leq j\leq \rep(m_2)-1,\\[6pt]
% \frac{3}{p^t} & \text{if }  \rep(m_2)\leq j\leq \rep(m_1)-1,\\[6pt]
% \frac{4}{p^t} & \text{if }  \rep(m_1)\leq j\leq p-1.\\[6pt]
\end{cases} 
&& \textup{and} &&
\gamma:=
\begin{cases}
1 & \text{if } 0\leq j\leq \rep(m_4)-1, \\[6pt]
\frac{1}{p} & \text{if }  \rep(m_4) \leq j\leq \rep(m_3)-1,\\[6pt]
\frac{1}{p^2} & \text{if }  \rep(m_3)\leq j\leq \rep(m_2)-1.\\[6pt]
\end{cases}
\end{align*}
Then using Lemmas \ref{lem_SumGammaG1} and \ref{lem_SumGammaG2} we see that, for $j \leq \rep(m_2)-1$,
\begin{align}\label{for_4G2Ap}
A(j) 
\equiv \sum_{k=1}^{4} \Bigl(H_{\rep(m_k)-1+j}^{(1)} - H_{j}^{(1)}\Bigr)- \delta_1
\pmod{p}
\end{align}
and
\begin{multline}\label{for_4G2Bp}
\sum_{k=1}^{4} \Bigl(G_1\left(m_k+j\right)^2-G_2\left(m_k+j\right)
-G_1\left(1+j\right)^2 +G_2\left(1+j\right)\Bigr)\\
\equiv \sum_{k=1}^{4} \Biggl(H_{\rep(m_k)+j-1}^{(2)}- \hspace{1pt} H_{j}^{(2)}\Biggr)-\delta_2
\pmod{p}.
\end{multline}
\noindent Using Lemma \ref{lem_ProdGammap} we get that, for $j \leq \rep(m_2)-1$,
\begin{equation*}
\prod_{k=1}^{4}  \biggfp{m_k+j}\\
\equiv \left[\,\prod_{k=1}^{4} \left(\rep(m_k)+j-1\right)! \,(-1)^{\rep(m_k)+j}\right] \cdot \gamma
\pmod{p}
\end{equation*}
and by Proposition \ref{prop_pGamma} (2) we see that
\begin{align*}
\prod_{k=1}^{4} \biggfp{m_k} &
= \prod_{k=1}^{2} (-1)^{\rep(m_k)}=(-1)^{\rep(m_1)+\rep(m_2)}=\pm1.
\end{align*}

\noindent 
Substituting for $A(j)$, $B(j)$, $\prod_{k=1}^{4} \biggfp{m_k+j}$ and $\prod_{k=1}^{4} \biggfp{m_k}$ modulo $p$ into (\ref{for_Xj}) we have
\begin{align*}
\pm X(j)
\equiv &
\sum_{j=0}^{\rep(m_4)-1} 
\Biggl[\, \prod_{k=1}^{4} \ph{j+1}{\rep(m_k)-1} \Biggr]
\Biggl[1+2j \sum_{k=1}^{4} \Bigl(H_{\rep(m_k)-1+j}^{(1)}- H_{j}^{(1)}\Bigr)
\\ & 
+\frac{j^2}{2} \Biggl(\left(\sum_{k=1}^{4} \Bigl(H_{\rep(m_k)-1+j}^{(1)}- H_{j}^{(1)}\Bigr)\right)^2
-\sum_{k=1}^{4} \Bigl(H_{\rep(m_k)-1+j}^{(2)}- H_{j}^{(2)}\Bigr)
\Biggr)
\Biggr]
\\ &
+\sum_{j=\rep(m_4)}^{\rep(m_3)-1}
\Biggl[\, \prod_{k=1}^{4} \ph{j+1}{\rep(m_k)-1} \Biggr]
\Biggl[ \frac{1}{p} \Biggr]
\Biggl[2j+j^2 \sum_{k=1}^{4} \Bigl(H_{\rep(m_k)-1+j}^{(1)}- H_{j}^{(1)}\Bigr)
 -\frac{j^2}{p} \Biggr]
\\ &
+\sum_{\rep(m_3)}^{\rep(m_2)-1} 
\Biggl[\, \prod_{k=1}^{4} \ph{j+1}{\rep(m_k)-1} \Biggr]
\Biggl[ \frac{1}{p^2} \Biggr]
\Biggl[ j^2 \Biggr]
\pmod{p}.
\end{align*}

We can simplify this expression by combining the three terms into one single summation.
For $\rep(m_4) \leq j \leq \rep(m_3) -1$ we note that $\prod_{k=1}^{4} \ph{j+1}{\rep(m_k)-1}  \in p\mathbb{Z}_p$. Also, we see from (\ref{for_4G2Ap}) and (\ref{for_4G2Bp}) that $\sum_{k=1}^{4} \Bigl(H_{\rep(m_k)-1+j}^{(1)}- H_{j}^{(1)}\Bigr)-\frac{1}{p} \in \mathbb{Z}_p$ and $\sum_{k=1}^{4} \Bigl(H_{\rep(m_k)-1+j}^{(2)}- H_{j}^{(2)}\Bigr)-\frac{1}{p^2} \in \mathbb{Z}_p$, for $j$ in the same range.
Similarly, for $\rep(m_3) \leq j \leq \rep(m_2)- 1$ we have that $\prod_{k=1}^{4} \ph{j+1}{\rep(m_k)-1}  \in p^2\mathbb{Z}_p$, $\sum_{k=1}^{4} \Bigl(H_{\rep(m_k)-1+j}^{(1)}- H_{j}^{(1)}\Bigr)-\frac{2}{p} \in \mathbb{Z}_p$ and $\sum_{k=1}^{4} \Bigl(H_{\rep(m_k)-1+j}^{(2)}- H_{j}^{(2)}\Bigr)-\frac{2}{p^2} \in \mathbb{Z}_p$.
Consequently,
\begin{multline*}
\pm X(j)
\equiv
\sum_{j=0}^{\rep(m_2)-1}  
\Biggl[\, \prod_{k=1}^{4} \ph{j+1}{\rep(m_k)-1} \Biggr]
\Biggl[1+2j\sum_{k=1}^{4} \Bigl(H_{\rep(m_k)-1+j}^{(1)} -H_{j}^{(1)}\Bigr)
\\
+\frac{j^2}{2} \Biggl(\left(\sum_{k=1}^{4} \Bigl(H_{\rep(m_k)-1+j}^{(1)} - H_{j}^{(1)}\Bigr)\right)^2
-\sum_{k=1}^{4} \Bigl(H_{\rep(m_k)-1+j}^{(2)} - H_{j}^{(2)}\Bigr)
\Biggr)
\Biggr]
\pmod{p}.
\end{multline*}
\noindent Note we can extend the upper limit of this sum to $j=p-1$ as $\prod_{k=1}^{4} \ph{j+1}{\rep(m_k)-1} \in p^3\mathbb{Z}_p$ for $\rep(m_2) \leq j\leq p-1$.
By Lemmas \ref{lem_P} and \ref{lem_Q} with $n=4$ and $a_i=\rep(m_i)-1$ for $1 \leq i \leq 4$ we get that
\begin{equation}\label{for_4G2X0}
X(j)\equiv \pm(1-1) \equiv 0 \pmod{p}.
\end{equation}

Accounting for (\ref{for_4G2X0}) in (\ref{Resid1_4G2}), via (\ref{Resid1_4G2_split}) and (\ref{for_Xj}), means we need only show
\begin{align}\label{Resid2_4G2}
\notag &\sum_{j=0}^{\rep(m_4)-1} 
\left[ \, \prod_{k=1}^{4} \frac{ \biggfp{m_k+j}}
{\biggfp{m_k}\, j!} \right]
\Bigl[1+jA(j)\Bigr]
+ \sum_{j=\rep(m_4)}^{\left\lfloor \frac{r(p-1)}{d} \right\rfloor} 
\left[ \, \prod_{k=1}^{4} \frac{ \biggfp{m_k+j}}
{\biggfp{m_k}\, j!} \right]
\Bigl[j\Bigr]\\[3pt] \notag
& - \sum_{j=\left\lfloor \frac{r(p-1)}{d} \right\rfloor+1}^{\rep(m_3)-1} 
\left[ \, \prod_{k=1}^{4} \frac{ \biggfp{m_k+j}}
{\biggfp{m_k}\, j!} \right]
\Bigl[\tfrac{1}{d}\Bigr]
+p \left\{
\sum_{j=\rep(m_4)}^{\left\lfloor \frac{r(p-1)}{d} \right\rfloor} 
\left[ \, \prod_{k=1}^{4} \frac{ \biggfp{m_k+j}}
{\biggfp{m_k}\, j!} \right]
\Bigl[\Bigl(\tfrac{1}{d}\Bigr)\Bigl(1+jA(j)\Bigr)\Bigr]
\right. \\[3pt] \notag &\left. 
+\sum_{j=\left\lfloor \frac{r(p-1)}{d} \right\rfloor+1}^{\rep(m_3)-1} 
\left[ \, \prod_{k=1}^{4} \frac{ \biggfp{m_k+j}}
{\biggfp{m_k}\, j!} \right]
\Bigl[\tfrac{r}{d}-j-j^2A(j)\Bigr]
+\sum_{j=\rep(m_3)}^{\left\lfloor (d-r)\frac{p-1}{d} \right\rfloor} 
\left[ \, \prod_{k=1}^{4} \frac{ \biggfp{m_k+j}}
{\biggfp{m_k}\, j!} \right]
\Bigl[j\Bigl(\tfrac{1}{d}+\tfrac{r}{d}\Bigr)\Bigr]
\right. \\[3pt]  &\left.
- \sum_{j=\left\lfloor (d-r)\frac{p-1}{d} \right\rfloor+1}^{\rep(m_2) - 1} 
\left[ \, \prod_{k=1}^{4} \frac{ \biggfp{m_k+j}}
{\biggfp{m_k}\, j!} \right]
\Bigl[\Bigl(\tfrac{1}{d}\Bigr) \Bigl(\tfrac{r}{d}\Bigr)+j^2\Bigr]
\right\} 
\equiv s(p)
\pmod{p^2}.
\end{align}

\noindent We now convert these remaining terms to an expression involving binomial coefficients and harmonic sums and then use the results of Section 3.3 to simplify them. First we define 
\begin{align*}
\Bin \hspace{1pt}(j) &:=  \bin{\rep(m_1)-1+j}{j} \bin{\rep(m_1)-1}{j}  \bin{\rep(m_2)-1+j}{j} \bin{\rep(m_2)-1}{j}\fc\\[15pt]
\mathcal{H}(j)&:= H_{\rep(m_1)-1+j}^{(1)}+H_{\rep(m_1)-1-j}^{(1)}+H_{\rep(m_2)-1+j}^{(1)}+H_{\rep(m_2)-1-j}^{(1)}-4H_j^{(1)},\\[15pt]
\mathcal{A}(j)&:= \Bigl(\rep(m_1)-m_1\Bigr)\left(H_{\rep(m_1)-1+j}^{(1)}-H_{\rep(m_{4})-1+j}^{(1)}\right)\\[6pt]
&  \qquad \qquad \qquad \qquad\qquad \quad
+ \Bigl(\rep(m_2)-m_2\Bigr)\left(H_{\rep(m_2)-1+j}^{(1)}-H_{\rep(m_{3})-1+j}^{(1)}\right),
\intertext{and}
\mathcal{B}(j)&:= \Bigl(\rep(m_1)-m_1\Bigr)\left(H_{\rep(m_1)-1+j}^{(2)}-H_{\rep(m_{4})-1+j}^{(2)}\right)\\[6pt]
&  \qquad \qquad \qquad \qquad\qquad \quad
+ \Bigl(\rep(m_2)-m_2\Bigr)\left(H_{\rep(m_2)-1+j}^{(2)}-H_{\rep(m_{3})-1+j}^{(2)}\right).
\end{align*}

\noindent By Lemma \ref{lem_ProdGammap2} we see that for $j<\rep({m_2})$
\begin{align}\label{for_4G1Binp2}
\notag\prod_{k=1}^{4} \frac{\biggfp{m_k+j}}{\biggfp{m_k} j!}
&\equiv
\Biggl[ \prod_{i=1}^{2} \bin{\rep(m_i)-1+j}{j} \bin{\rep(m_i)-1}{j}\Biggr]\cdot \alpha \\[6pt]
\notag& \qquad \cdot \left[1-\sum_{i=1}^{2}\Bigl(\rep(m_i)-m_i\Bigr)\left(H_{\rep(m_i)-1+j}^{(1)}-H_{\rep(m_{5-i})-1+j}^{(1)}-\beta_i\right)\right]\\[6pt]
&\equiv \Bin \hspace{1pt}(j) \cdot \alpha \cdot \left[1-\mathcal{A}(j) + \sum_{i=1}^{2}\Bigl(\rep(m_i)-m_i\Bigr) \beta_i\right]
\pmod{p^2},
\end{align}
where 
\begin{align*}
\alpha=
\begin{cases}
1 & \text{if } 0\leq j \leq \rep(m_4)-1,\\
\frac{1}{p} & \text{if } \rep(m_4) \leq j \leq \rep(m_3)-1,\\
\frac{1}{p^2} & \text{if } \rep(m_3) \leq j < \rep(m_2),
\end{cases}
&& \textup{and} &&
\beta_i=
\begin{cases}
0 & \text{if } 0\leq j \leq \rep(m_{5-i})-1,\\
\frac{1}{p} & \text{if } \rep(m_{5-i}) \leq j < \rep(m_i).
\end{cases}
\end{align*}

\noindent Again for $j<\rep({m_2})$, Lemma \ref{lem_SumGammap2} gives us 
\begin{align}\label{for_4G1Hp2}
\notag A(j)
&\equiv \sum_{i=1}^{2} 
\left( H_{\rep(m_i)-1+j}^{(1)}+H_{\rep(m_i)-1-j}^{(1)}-2\hspace{1pt} H_{j}^{(1)}\right)-\alpha^{\prime}\\
\notag &\qquad\qquad+\sum_{i=1}^{2} \left(\rep(m_i)-m_i\right) 
\left(H_{\rep(m_i)-1+j}^{(2)}-H_{\rep(m_{5-i})-1+j}^{(2)}-\beta_{i}^{\prime}\right)\\
&\equiv \mathcal{H}(j) -\alpha^{\prime} + \mathcal{B}(j) - \sum_{i=1}^{2} \left(\rep(m_i)-m_i\right) \beta_{i}^{\prime}
\pmod{p^2}
\end{align}
where 
\begin{align*}
\alpha^{\prime}=
\begin{cases}
0 & \text{if } 0\leq j \leq \rep(m_4)-1,\\
\frac{1}{p} & \text{if } \rep(m_4) \leq j \leq \rep(m_3)-1,\\
\frac{2}{p} & \text{if } \rep(m_3) \leq j < \rep(m_2),\;
\end{cases}
&& \textup{and} &&
\beta_{i}^{\prime}=
\begin{cases}
0 & \text{if } 0\leq j \leq \rep(m_{5-i})-1,\\
\frac{1}{p^2} & \text{if } \rep(m_{5-i}) \leq j < \rep(m_i).
\end{cases}
\end{align*}

\noindent Reducing (\ref{for_4G1Binp2}) and (\ref{for_4G1Hp2}) modulo $p$ respectively we see that for $j<\rep(m_2)$,
\begin{align}\label{for_4G1Binp}
\prod_{k=1}^{4} 
\frac{ \biggfp{m_k+j}}{\biggfp{m_k} j!}
&\equiv \Bin \hspace{1pt}(j) \cdot \alpha 
\pmod{p}
\end{align}
and\\[-24pt]
\begin{align}\label{for_4G1Hp}
A(j) \equiv \mathcal{H}(j) -\alpha^{\prime} 
\pmod{p},
\end{align}
as $\rep(m_i)-m_i \in p \mathbb{Z}_p$.
Therefore, using (\ref{for_4G1Binp2}),  (\ref{for_4G1Hp2}),  (\ref{for_4G1Binp}) and (\ref{for_4G1Hp}),  we see that (\ref{Resid2_4G2}) is equivalent to
\begin{multline}\label{Resid3_4G2}
\sum_{j=0}^{\rep(m_4)-1} 
\Bin \hspace{1pt}(j) \Bigl[1-\mathcal{A}(j) \Bigr]
\Bigl[1+j \mathcal{H}(j)  + j \mathcal{B}(j) \Bigr]
+ \sum_{j=\rep(m_4)}^{\left\lfloor \frac{r(p-1)}{d} \right\rfloor} 
 \Bigl[\tfrac{j \Bin(j)}{p}\Bigr] \Bigl[1-\mathcal{A}(j) + \tfrac{\rep(m_1)-m_1}{p}\Bigr]
 \\ 
- \sum_{j=\left\lfloor \frac{r(p-1)}{d} \right\rfloor +1}^{\rep(m_3) -1} 
 \Bigl[\tfrac{\Bin(j)}{dp}\Bigr] \Bigl[1-\mathcal{A}(j) + \tfrac{\rep(m_1)-m_1}{p}\Bigr]
+p \left\{
\sum_{j=\rep(m_4)}^{\left\lfloor \frac{r(p-1)}{d} \right\rfloor} 
 \Bigl[\tfrac{\Bin(j)}{dp}\Bigr]
\Bigl[1+j\mathcal{H}(j)-\tfrac{j}{p}\Bigr]
\right. \\ \shoveleft{\left. 
+\sum_{j=\left\lfloor \frac{r(p-1)}{d} \right\rfloor +1}^{\rep(m_3) -1}
 \Bigl[\tfrac{\Bin(j)}{p}\Bigr] \Bigl[\tfrac{r}{d}-j-j^2 \mathcal{H}(j)+\tfrac{j^2}{p}\Bigr]
+\sum_{j=\rep(m_3)}^{\left\lfloor (d-r)\frac{p-1}{d} \right\rfloor} 
\Bigl[\tfrac{\Bin(j)}{p^2}\Bigr]
\Bigl[j \Bigl(\tfrac{1}{d}+\tfrac{r}{d}\Bigr)\Bigr]
\right.} \\ \left. 
-\sum_{j=\left\lfloor (d-r)\frac{p-1}{d} \right\rfloor+1}^{\rep(m_2) - 1}
 \Bigl[\tfrac{\Bin(j)}{p^2}\Bigr]
\Bigl[\Bigl(\tfrac{1}{d}\Bigr)\Bigl(\tfrac{r}{d}\Bigr)+j^2\Bigr]
\right\} 
\equiv s(p) 
\pmod{p^2}.
\end{multline}

\noindent We now consider
\begin{equation*}
\sum_{j=0}^{\rep(m_2)- 1}
\Bin \hspace{1pt}(j) \Bigl[1+j \mathcal{H}(j) \Bigr]
+\sum_{j=0}^{\rep(m_2) - 1}
\Bin \hspace{1pt}(j)
\Bigl[ j \mathcal{B}(j) -\mathcal{A}(j) - j \mathcal{A}(j) \mathcal{H}(j) \Bigr]
\pmod{p^2}.\\
\end{equation*}
\noindent For $0\leq j \leq \rep(m_4)-1$ we see that
\begin{multline}\label{for_Bin1_m4}
\Bin \hspace{1pt}(j) \Bigl[1-\mathcal{A}(j) \Bigr]
\Bigl[1+j \mathcal{H}(j)  + j \mathcal{B}(j) \Bigr]\\[6pt]
\equiv
\Bin \hspace{1pt}(j) \Bigl[1+j \mathcal{H}(j)+ j \mathcal{B}(j) -\mathcal{A}(j) - j \mathcal{A}(j) \mathcal{H}(j) \Bigr]
\pmod{p^2},
\end{multline}
as $ \mathcal{A}(j)  \mathcal{B}(j) \in p^2 \mathbb{Z}_p$ for such $j$.
\noindent For $\rep(m_4) \leq j \leq \rep(m_3)-1$ we note the following facts:
\begin{itemize}
\emptyline
\item $\Bin \hspace{1pt}(j) \in p \mathbb{Z}_p$, which we can see from (\ref{for_4G1Binp});
\emptyline
\item $\mathcal{H}(j) -\frac{1}{p} \in \mathbb{Z}_p$, from (\ref{for_4G1Hp});
\emptyline
\item $\mathcal{A}(j)- (\rep(m_1)-m_1)\frac{1}{p} \in p \mathbb{Z}_p$; and
\emptyline
\item $\mathcal{B}(j)- (\rep(m_1)-m_1)\frac{1}{p^2} \in p \mathbb{Z}_p$.
\emptyline
\end{itemize}
The last two properties come directly from their definitions, noting that $\rep(m_i)-m_i \in p \mathbb{Z}_p$ by definition. Therefore, for $j$ in this range, we have
\begin{align*}
\Bin \hspace{1pt}(&j) \Bigl[1+j \mathcal{H}(j)+ j \mathcal{B}(j) -\mathcal{A}(j) - j \mathcal{A}(j) \mathcal{H}(j) \Bigr] \\[9pt]
&\equiv
\Bin \hspace{1pt}(j) \Bigl\{1+j \mathcal{H}(j) +j \tfrac{\rep(m_1)-m_1}{p^2} - \tfrac{\rep(m_1)-m_1}{p}-
j\left[\tfrac{1}{p} \left(\mathcal{A}(j)- \tfrac{\rep(m_1)-m_1}{p}\right)
\right.\\[9pt] &\left. \qquad \qquad 
 \qquad \qquad \qquad \qquad \qquad
\quad \; \; \,
+\tfrac{\rep(m_1)-m_1}{p} \left(\mathcal{H}(j)-\tfrac{1}{p}\right) +\tfrac{\rep(m_1)-m_1}{p^2}\right]\Bigr\}\\[9pt]
&\equiv
\Bin \hspace{1pt}(j) \Bigl[\Bigl(1+j \mathcal{H}(j)\Bigr)\left(1- \tfrac{\rep(m_1)-m_1}{p}\right) +2j \tfrac{\rep(m_1)-m_1}{p^2} -
\tfrac{j}{p} \mathcal{A}(j)\Bigr]
\pmod{p^2}
\end{align*}
Recall that $m_1:=\frac{a}{d}$, where $p \equiv a \pmod{d}$ with $a<d$, and $\rep(m_1)=p-\lfloor \frac{p-1}{d} \rfloor$. Hence
\begin{equation}\label{for_4G2RepOverp}
\rep(m_1)-m_1 = p- \lfloor \tfrac{p-1}{d} \rfloor - \tfrac{a}{d}
=  p- \tfrac{p-a}{d} - \tfrac{a}{d} = p\, \bigl(1 - \tfrac{1}{d}\bigr).
\end{equation}
Therefore, for $\rep(m_4) \leq j \leq \rep(m_3)-1$,
\begin{align}\label{for_Bin1_m3}
\notag \Bin \hspace{1pt}(j)& \Bigl[1+j \mathcal{H}(j)+ j \mathcal{B}(j) -\mathcal{A}(j) - j \mathcal{A}(j) \mathcal{H}(j) \Bigr] \\[9pt]
&\equiv
\Bin \hspace{1pt}(j) \Bigl[\Bigl(\tfrac{1}{d}\Bigr)\Bigl(1+j\mathcal{H}(j)-\tfrac{j}{p}\Bigr)
+\Bigl(\tfrac{j}{p}\Bigr)\Bigl(1-\mathcal{A}(j)+\tfrac{\rep(m_1)-m_1}{p}\Bigr)\Bigr]
\pmod{p^2}.
\end{align}
\noindent Similarly, for $\rep(m_3) \leq j \leq \rep(m_2)-1$,
\begin{align}\label{for_Bin1_m2}
\Bin \hspace{1pt}(j) &\Bigl[1+j \mathcal{H}(j)+ j \mathcal{B}(j) -\mathcal{A}(j) - j \mathcal{A}(j) \mathcal{H}(j) \Bigr] 
\equiv
\Bin \hspace{1pt}(j) \Bigl[\tfrac{1}{p}\Bigr]
\Bigl[j \Bigl(\tfrac{1}{d}+\tfrac{r}{d}\Bigr)\Bigr]
\pmod{p^2}.
\end{align}
Accounting for (\ref{for_Bin1_m4}), (\ref{for_Bin1_m3}) and (\ref{for_Bin1_m2}) in (\ref{Resid3_4G2}), it now suffices to show
\begin{multline}\label{Resid4_4G2}
\sum_{j=0}^{\rep(m_2) - 1}
\Bin \hspace{1pt}(j) \Bigl[1+j \mathcal{H}(j) \Bigr]
+\sum_{j=0}^{\rep(m_2) - 1}
\Bin \hspace{1pt}(j)
\Bigl[ j \mathcal{B}(j) -\mathcal{A}(j) - j \mathcal{A}(j) \mathcal{H}(j) \Bigr]\\
-\sum_{j=\left\lfloor \frac{r(p-1)}{d} \right\rfloor+1}^{\rep(m_3)-1} 
\Bin \hspace{1pt}(j)
\biggl\{\Bigl[\tfrac{1}{d}+j\Bigr] \Bigl[\Bigl(\tfrac{1}{p}\Bigr) \left(1-\mathcal{A}(j) + \tfrac{\rep(m_1)-m_1}{p}\right) +1+j\mathcal{H}(j)-\tfrac{j}{p}\Bigr]
-\tfrac{r}{d}
\biggr\}
\\
-\sum_{j=\left\lfloor (d-r)\frac{p-1}{d} \right\rfloor+1}^{\rep(m_2) - 1}
\Bin \hspace{1pt}(j) \Bigl[\tfrac{1}{p}\Bigr]
\Bigl[j^2+\Bigl(\tfrac{1}{d}\Bigr)\Bigl(\tfrac{r}{d}\Bigr)+j \Bigl(\tfrac{1}{d}+\tfrac{r}{d}\Bigr)\Bigr]
\equiv s(p) 
\pmod{p^2}.
\end{multline}

\noindent Recall that the third sum above is vacuous unless $p\equiv r \pmod{d}$ and $r^2\equiv 1\pmod{d}$ or $p\equiv d-r \pmod{d}$ and $r^2\equiv -1\pmod{d}$. In these cases the sum is over one value of $j=\rep(m_3)-1 = \bigl\lfloor \tfrac{r(p-1)}{d} \bigr\rfloor +1$ and so $\left(\tfrac{1}{d} +j\right) = \tfrac{rp}{d}$. Also $\rep(m_1)-m_1=p\left(1-\tfrac{1}{d}\right)$. So we get 
\begin{multline*}
\sum_{j=\left\lfloor \frac{r(p-1)}{d} \right\rfloor+1}^{\rep(m_3)-1} 
\Bin \hspace{1pt}(j)
\biggl[\Bigl[\tfrac{1}{d}+j\Bigr] \Bigl[\Bigl(\tfrac{1}{p}\Bigr) \left(1-\mathcal{A}(j) + \tfrac{\rep(m_1)-m_1}{p}\right) +1+j\mathcal{H}(j)-\tfrac{j}{p}\Bigr]
-\tfrac{r}{d}
\biggr]\\
=
\sum_{j=\left\lfloor \frac{r(p-1)}{d} \right\rfloor+1}^{\rep(m_3)-1} 
\Bin \hspace{1pt}(j)
\Bigl[\tfrac{rpj}{d}\mathcal{H}(j)-\tfrac{r}{d}\mathcal{A}(j)+\tfrac{r}{d}\tfrac{d-1}{d}+\tfrac{rp}{d}-\tfrac{rj}{d}
\Bigr].
\end{multline*}

\noindent Note $\Bin \hspace{1pt}(j) \in p \mathbb{Z}_p$ for $j=\rep(m_3)-1$. Thus $\Bin \hspace{1pt}(j) \hspace{1pt}p\hspace{1pt} \mathcal{H}(j)\equiv \Bin \hspace{1pt}(j) \hspace{1pt}p \left(\frac{1}{p}\right) \pmod{p^2}$ and $\Bin \hspace{1pt}(j) \mathcal{A}(j) \equiv \Bin \hspace{1pt}(j) \frac{\rep(m_1)-m_1}{p} \equiv \Bin \hspace{1pt}(j) \frac{d-1}{d}\pmod{p^2}$ for this value of $j$, and 
\begin{equation*}
\sum_{j=\left\lfloor \frac{r(p-1)}{d} \right\rfloor+1}^{\rep(m_3)-1} 
\Bin \hspace{1pt}(j)
\Bigl[\tfrac{rpj}{d}\mathcal{H}(j)-\tfrac{r}{d}\mathcal{A}(j)+\tfrac{r}{d}\tfrac{d-1}{d}+\tfrac{rp}{d}-\tfrac{rj}{d}
\Bigr]
\equiv
\sum_{j=\left\lfloor \frac{r(p-1)}{d} \right\rfloor+1}^{\rep(m_3)-1} 
\Bin \hspace{1pt}(j)
\Bigl[\tfrac{rp}{d}
\Bigr]
\equiv 0
\pmod{p^2}.
\end{equation*}
\noindent So the third term of (\ref{Resid4_4G2}) vanishes modulo $p^2$. Next we examine the last term of (\ref{Resid4_4G2}),
\begin{equation*}
\sum_{j=\left\lfloor (d-r)\frac{p-1}{d} \right\rfloor+1}^{\rep(m_2) - 1}
\Bin \hspace{1pt}(j) \Bigl[\tfrac{1}{p}\Bigr]
\Bigl[j^2+\Bigl(\tfrac{1}{d}\Bigr)\Bigl(\tfrac{r}{d}\Bigr)+j \Bigl(\tfrac{1}{d}+\tfrac{r}{d}\Bigr)\Bigr]
=\sum_{j=\left\lfloor (d-r)\frac{p-1}{d} \right\rfloor+1}^{\rep(m_2) - 1}
\Bin \hspace{1pt}(j) \Bigl[\tfrac{1}{p}\Bigr]
\Bigl[\Bigl(j+\tfrac{1}{d}\Bigr)\Bigl(j+\tfrac{r}{d}\Bigr)\Bigr]
\end{equation*}
modulo $p^2$. Recall that this sum is vacuous unless $p\equiv d-1\pmod{d}$, $p \equiv r \pmod{d}$ and $r^2\equiv-1\pmod{d}$ or $p \equiv d-r \pmod{d}$ and $r^2\equiv1\pmod{d}$. In these cases the limits of summation are equal and the sum is over one value of $j=\rep(m_2)-1=p-r\lfloor \tfrac{p-1}{d} \rfloor - \left\lfloor \frac{ar}{d} \right\rfloor -1$. Now
\begin{equation*}
p-r\lfloor \tfrac{p-1}{d}\rfloor - \left\lfloor \tfrac{ar}{d} \right\rfloor -1
=p-r\left(\tfrac{p-a}{d}\right)- \left\lfloor \tfrac{ar}{d} \right\rfloor -1
=p\left(1-\tfrac{r}{d}\right) + \langle \tfrac{ar}{d} \rangle-1.
\end{equation*}
Thus if $p\equiv d-1 \pmod{d}$,
\begin{equation*}
j=p\left(1-\tfrac{r}{d}\right)+\langle r-\tfrac{r}{d}\rangle -1\\
=p\left(1-\tfrac{r}{d}\right)+1-\tfrac{r}{d}-1
=p\left(1-\tfrac{r}{d}\right)-\tfrac{r}{d}.
\end{equation*}
Then
$$\left(j+\tfrac{r}{d}\right) = p\left(1-\tfrac{r}{d}\right) \in p \mathbb{Z}_p$$
and
$$\left(j+\tfrac{1}{d}\right) = p\left(1-\tfrac{r}{d}\right) -\tfrac{r}{d} +\tfrac{1}{d}\in  \mathbb{Z}_p.$$
\noindent If $p\equiv r \pmod{d}$ and $r^2 \equiv -1 \pmod{d}$,
$$j=p\left(1-\tfrac{r}{d}\right)-\langle \tfrac{r^2}{d} \rangle -1=p\left(1-\tfrac{r}{d}\right)+\tfrac{d-1}{d}-1=p\left(1-\tfrac{r}{d}\right)-\tfrac{1}{d}.$$
\noindent If $p\equiv d-r \pmod{d}$ and $r^2 \equiv 1 \pmod{d}$,
$$j=p\left(1-\tfrac{r}{d}\right)+\langle \tfrac{dr-r^2}{d} \rangle -1=p\left(1-\tfrac{r}{d}\right)+\tfrac{d-1}{d}-1=p\left(1-\tfrac{r}{d}\right)-\tfrac{1}{d}.$$
Then, in both cases,
$$\left(j+\tfrac{1}{d}\right) = p\left(1-\tfrac{r}{d}\right) \in p \mathbb{Z}_p$$
and
$$\left(j+\tfrac{r}{d}\right) = p\left(1-\tfrac{r}{d}\right) -\tfrac{1}{d} +\tfrac{r}{d}\in  \mathbb{Z}_p.$$
\noindent Therefore,\\
\begin{equation*}
\sum_{j=\left\lfloor (d-r)\frac{p-1}{d} \right\rfloor+1}^{\rep(m_2) - 1}
\Bin \hspace{1pt}(j) \Bigl[\tfrac{1}{p}\Bigr]
\Bigl[\Bigl(j+\tfrac{1}{d}\Bigr)\Bigl(j+\tfrac{r}{d}\Bigr)\Bigr]
\equiv 0
\pmod{p^2},
\emptyline
\end{equation*}
as $\Bin \hspace{1pt}(j) \in p^2\mathbb{Z}_p$ for $j=\rep(m_2)-1$.

Finally, we examine the two remaining terms of (\ref{Resid4_4G2}).  Taking $m=\rep(m_1)-1$ and $n=\rep(m_2)-1$ in Corollary \ref {Cor_BinHarId1} we get that
\begin{equation*}
\sum_{j=0}^{\rep(m_2) - 1}
\Bin \hspace{1pt}(j) \Bigl[1+j \mathcal{H}(j) \Bigr] \equiv s(p) 
\pmod{p^2},
\end{equation*}
and taking $l=p$, $m=\rep(m_1)-1$, $n=\rep(m_2)-1$, $c_2=\rep(m_1)-m_1$ and $c_1=\rep(m_2)-m_2$ in Corollary \ref {Cor_BinHarId2} we get
\begin{equation*}
\sum_{j=0}^{p-\left\lfloor \frac{p-1}{d_2} \right\rfloor - 1}
\Bin \hspace{1pt}(j)
\Bigl[ j \mathcal{B}(j) -\mathcal{A}(j) - j \mathcal{A}(j) \mathcal{H}(j) \Bigr]
\equiv 0 \pmod{p^2}
\end{equation*}
as required.
\end{proof}

\end{document}